\documentclass[a4paper,11pt]{amsart}

\usepackage[utf8]{inputenc}
\usepackage{mathrsfs}
\usepackage{amsfonts, amssymb, amsmath, amsthm, amsxtra, enumerate}
\usepackage[english]{babel}
\usepackage[active]{srcltx}
\usepackage{multirow}
\usepackage{latexsym}
\usepackage{lscape}
\usepackage{color}
\usepackage{url}
\usepackage[all]{xy}
\usepackage[colorlinks=true,pdfauthor={Ariel Pacetti, Nicolas
Sirolli},pdftitle={Computing ideal classes representatives in
quaternion algebras},pdfkeywords={Quaternion algebras, ideal classes
representatives, Bass orders}]{hyperref}

\newcommand{\diag}{\operatorname{diag}}
\newcommand{\Q}{\mathbb{Q}}
\newcommand{\OO}{\mathcal{O}}

\newcommand{\E}{\mathcal{E}}
\newcommand{\tr}{\operatorname{Tr}}

\newcommand{\pp}{\mathfrak{p}}

\newcommand{\qq}{\mathfrak{q}}
\newcommand{\Kp}{\mathbb{K}_\mathfrak{p}}
\newcommand{\pip}{\pi_{\mathfrak{p}}}
\newcommand{\ida}{\Psi_{\O'}^\O(I)}
\newcommand{\cl}{\mathit{Cl}}

\numberwithin{equation}{section}
\numberwithin{table}{section}
\numberwithin{figure}{section}

\newtheorem*{thm*}{Theorem}
\newtheorem{prop}[equation]{Proposition}
\newtheorem{coro}[equation]{Corollary}

\newtheorem*{conj*}{Conjecture}
\newtheorem{lemma}[equation]{Lemma}
\newtheorem*{thma}{Theorem A}
\newtheorem*{thmb}{Theorem B}

\theoremstyle{remark}
\newtheorem{remark}[equation]{Remark}
\newtheorem{exm}[equation]{Example}
\newtheorem{algo}[equation]{Algorithm}

\theoremstyle{definition}
\newtheorem{defi}[equation]{Definition}

\usepackage[mathscr]{euscript}
\usepackage{latexsym}

\usepackage[all]{xy}

\usepackage{amssymb}
\usepackage{amsthm}
\usepackage{amsmath}
\usepackage{amsxtra}

.tfm
.tfm

\DeclareMathOperator{\End}{End}

\DeclareMathOperator{\sign}{sign}

\newcommand{\B}{\mathcal B}

\def\ZZ{\mathbb Z}

\def\<#1>{{\left\langle{#1}\right\rangle}}

\def\Q{{\mathbb Q}}             
\def\QQ{{\mathbb Q}}             

\let\kro\dkro

\def\O{R}           
\def\Op{\O_\pp}                 
\def\Oq{\O_\qq}                 
\def\Ol(#1){{\mathop{\O_l}(#1)}}                 
\def\Or(#1){{\mathop{\O_r}(#1)}}                 
\def\Orp(#1){{\mathop{\O_r}(\idp{#1})}}               
\def\Otern(#1){{\mathop{\O^0_r}(\id{a})}}    

\def\id#1{{\mathfrak{#1}}}      
\def\idp#1{{\id{#1}_\id{p}}}         
\def\kp{k_{\id{p}}}         

\def\T_#1(#2){{\mathop{\mathscr T}\nolimits_{#1}(\id{#2})}}
\def\TO(#1)_#2(#3){{\mathop{\mathscr T}\nolimits^{#1}_{#2}(\id{#3})}}

\def\A_#1(#2){{\mathop{\mathscr A}\nolimits_{#1}(\id{#2})}}
\def\Ax_#1(#2){{\mathop{\widetilde{\mathscr A}}\nolimits_{#1}(\id{#2})}}

\def\localconstant#1(#2){\varepsilon_{#1}({#2})}
\def\sign#1(#2){\epsilon_{#1}({#2})}

\begin{document}

\title{Computing ideal classes representatives\\in quaternion algebras}

\author{Ariel Pacetti}
\email{apacetti@dm.uba.ar}
\address{Departamento de Matem\'atica, Universidad de Buenos Aires - 
         Pabell\'on I, Ciudad Universitaria (C1428EGA), Buenos Aires,
         Argentina}
\author{Nicol\'as Sirolli}
\email{nsirolli@dm.uba.ar}
\address{Departamento de Matem\'atica, Universidad de Buenos Aires - 
         Pabell\'on I, Ciudad Universitaria (C1428EGA), Buenos Aires,
         Argentina}

\thanks{The first author was partially supported by PIP 2010-2012 GI and
UBACyT X867. The second author was partially supported by a CONICET PhD
Fellowship}


\begin{abstract} 
Let $K$ be a totally real number field and let $B$ be a totally definite
quaternion algebra over $K$. In this article, given a set of representatives for
ideal classes for a maximal order in $B$, we show how to construct in an
efficient way a set of representatives of ideal classes for any Bass order in
$B$. The algorithm does not require any knowledge of class numbers, and improves
the equivalence checking process by using a simple calculation with global
units. As an application, we compute ideal classes representatives for an order
of discriminant $30$ in an algebra over the real quadratic field $\Q[\sqrt{5}]$.
\end{abstract}

\maketitle

\renewcommand{\labelenumi}{(\alph{enumi})}
\renewcommand{\labelenumii}{\roman{enumii}.}

\section*{Introduction}

The theory of quaternion algebras over number fields plays a central role in
many computations related to modular forms. For example, orders in totally
definite quaternion algebras over totally real fields can be used to compute 
Hilbert modular forms,  as explained in \cite{Pizer} for classical modular forms
and in \cite{lassina} for Hilbert modular forms over totally real fields of even
degree. These methods require first to find a suitable order in such algebra,
and then compute representatives for the equivalence classes of its left ideals.
The purpose of this article is to compute both things in an efficient way, and
in a rather general setting, which includes Eichler orders and many others - for
example, the orders used in \cite{tornaria} to compute half-integral weight
modular forms in Shimura correspondence with modular forms of level $p^2$. 

Let $K$ be a number field and let $B$ be a quaternion algebra over $K$. When
computing ideal classes representatives, locally isomorphic orders in $B$ can be
regarded as equal, since two such orders have a connecting ideal, and
multiplication by this ideal gives a bijection between ideal classes
representatives for both orders. Hence, it is natural to group locally
isomorphic orders into \emph{genera}. Our first main result is the following
theorem.

\begin{thma}
 
There is an algorithm that, given a Bass order $\O$ in $B$, computes
suborders of $\O$ of any given genus.

\end{thma}

In particular, Theorem A allows us to calculate any Bass order in any quaternion
algebra, since by \cite{voight} we know how to obtain maximal orders in
this general setting.

The second main result concerns the computation of left ideal classes
representatives for Bass orders, assuming that $K$ is totally real and $B$ is
totally definite.

\begin{thmb}

There is an algorithm that, given a Bass order $\O$ in $B$ and a set of
representatives $S$ of left $\O$-ideal classes, computes left ideal classes
representatives for suborders of $\O$ of any given genus. Furthermore, the set
of norms of the computed ideals is the same as the set of norms of the ideals in
$S$. 

\end{thmb}

Hence, starting from a set of representatives for a maximal order (which can be
obtained following \cite{Pizer} or \cite{socrates} in certain particular cases,
and \cite{kirschmer} in the general setting), we can compute representatives for
any Bass order in $B$.

The algorithm is such that that the constructed ideals are
contained in the given ones. This allows to, in comparison to the methods
\textit{à la Pizer} (see, e.g., \cite{Pizer}, \cite{consani},
\cite{socrates}), avoid the repeated usage of norm forms
for checking equivalences between ideals (see \cite{Pizer}, Propositions 1.18
and 2.27), using this technique just once (see Remark
\ref{rmk:iterando_ideales}).

Bass orders can be described locally in terms of certain ternary quadratic
forms. The strategy for proving Theorems A and B is to reduce the situation to
the case of considering \emph{maximal} Bass suborders of $\O$. This allows to
construct both the desired suborder and its ideal classes representatives in
terms of
local computations related to the forms in correspondence with the orders. In
this special case, we also give a method to compute the ideal classes
representatives by global means.

The article is organized as follows. In the first section we give the
basic definitions that will be used throughout the article. In the second
section we prove Theorem A, first recalling the local description of Bass
orders. The third section is devoted to prove Theorem B. In the fourth section
we present an example of our algorithm: we show how to construct representatives
of ideal classes for an Eichler order of discriminant $30$ in the quaternion
algebra
$B$ over $\Q[\sqrt{5}]$ ramified at the two infinite places.

Throughout the article, in order to make the exposition clearer, we assume that
no dyadic primes occur in the discriminants of the orders considered. This case,
with the extra assumption that $2$ is inert in $K$, is treated separately in the
appendix.

\medskip

\noindent{\bf Acknowledgements:} We would like to thank Gonzalo
Tornar{\'\i}a and Lassi\-na Demb\'el\'e for the useful conversations we
held with them.

\section{Basic notions and notation} \label{sec:intro}

We start recalling some basic definitions and properties of quaternion
algebras that will be used during the paper. A more detailed exposition can be
found, for example, in \cite{vig} and \cite{kap}.

Let $\OO$ be a Dedekind domain, and let $K$ denote its fraction field. Let $\pp$
be a prime ideal of $\OO$. By $\OO_\pp$ we denote
the completion of $\OO$ at $\pp$, and we denote completions of other
objects in a similar way. By $v_\pp$ we denote the $\pp$-adic valuation on
$K_\pp$. The residue field $\OO_\pp/\pp\OO_\pp$ is denoted
by $\kp$, and by $\pip$ we denote an element of $\OO$ which is a local
uniformizer of $\pp\OO_\pp$.

Let $B$ be a \emph{quaternion algebra} over $K$, that is, a four dimensional,
central and simple $K$-algebra with unity. Then $B$ has a natural involution
$x\mapsto \bar x$, that induces the linear form \emph{(reduced) trace} given by
$\tr(x)=x+\bar x$ and the quadratic form \emph{(reduced) norm} given by
$N(x)=x\bar x$. The bilinear form corresponding to the latter is
$(x,y)\mapsto\tr(x\bar y)$.

A \emph{lattice} $\Lambda$ in $B$ is a finitely generated $\OO$-module
such that $\Lambda\otimes_\OO K\simeq B$. Given a lattice $\Lambda$, its
\emph{dual lattice} $\Lambda^\vee$ is defined by
\[
 \Lambda^\vee=\{x\in B:\tr(x\Lambda)\subseteq \OO\}.
\]

An \emph{order} is a lattice $\O$ which is also a subring with unity. Its
\emph{(reduced) discrimi\-nant} is the ideal $d(\O)\subseteq\OO$ whose square is
the ideal generated by $\{\det(\tr(x_i \bar{x_j})):x_1,\dots,x_4\in\O\}$.

Given a lattice $\Lambda$, the set
\[
\O_l(\Lambda)=\{x\in B: x\Lambda\subseteq\Lambda\}
\]
is an order called \emph{the left order of } $\Lambda$. The right
order is defined and denoted in a similar way. We define the \emph{inverse} of $
\Lambda$ by
\[
 \Lambda ^{-1}=\{x\in B:  \Lambda x \Lambda \subseteq \Lambda \}.
\]
We say that $ \Lambda $ is \emph{invertible} if
$ \Lambda  \Lambda ^{-1}=\Ol(\Lambda)$ and $ \Lambda ^{-1} \Lambda
=\Or(\Lambda)$. An order $\O$ is called  a \emph{Gorenstein} order
if every lattice $\Lambda$ such that $\O_l(\Lambda)=\O$ is invertible, and it is
called a \emph{Bass} order if every order containing it is a Gorenstein order.

Given two lattices $ \Lambda \supseteq \Lambda' $ in $B$, the \emph{index} of $
\Lambda' $ in $ \Lambda $ is the ideal $[ \Lambda : \Lambda' ]\subseteq\OO$
generated
by $\{\det(\phi):\phi\in\End_K(B),\,\phi( \Lambda )\subseteq \Lambda' \}$.

Let $\O$ be an order in $B$. A \emph{left $\O$-(invertible) ideal} is an
invertible lattice $ I $ such that $\Ol(I)=\O$. Two left $\O$-ideals $ I $ and
$ J $ are called \emph{equivalent} if there exists $x\in B^\times$
such that $ I = J x$, and the set of equivalence classes is denoted by $\cl(\O)
$. A left $\O$-ideal $ I $ is called
\emph{principal} if it is equivalent to $\O$, i.e., if there exists $x
\in B^\times$ such that $ I  = \O x$. A lattice $ I $ is
invertible if and only if $ I_\pp$ is principal for all $\pp$.

Let $\O,\O'$ be orders in $B$. We say that they are in the same
\emph{genus} if $\Op\simeq\Op'$ for all $\pp$. This is equivalent
to the existence of an ideal $I$ connecting $\O$ and $\O'$, i.e.,
such that $\Ol(I)=\O$ and $\Or(I)=\O'$.

\medskip

\noindent \textbf{Notation index}
\begin{itemize}

\item $\pp,\qq,\dots$: prime ideals of $\OO$.
\item $\Lambda,\Lambda',\dots$: lattices in $B$.
\item $\O,\O',\dots$: orders in $B$.
\item $\O^{\times,1}=\{x\in\O:N(x)=1\}$.
\item $I,J,\dots$: invertible lattices in $B$.
\item $\< a_1,\dots,a_n>$: the quadratic form $\sum_{i=1}^n a_i x_i^2$
\item $\diag(a_1,\dots,a_n)$: the diagonal matrix with $a_i$ as $(i,i)$
coefficient.

\end{itemize}

\section{Constructing suborders}

The aim of this section is to prove Theorem A. Its proof, together with a
precise description of the input of the algorithm,
will we given at the end of the section, once we have developed the necessary
tools.

The problem can be reduced to compute \emph{maximal} suborders of $\O$ in any
given genus. The index of a maximal suborder of a given order is known,
according to Corollary 1.11 of \cite{brz-ord}, which we recall here.

\begin{prop}\label{prop:e=1_o_2}
Let $\O$ be an order in $B$, and let $\O'$ be a maximal suborder of $\O$. Then,
there exists $\pp$ such that $[\O:\O']=\pp$ or $\pp^2$ and $\pp\O'\subseteq\O$.
\end{prop}

Hence, maximal suborders of a given order $\O$ can be obtained by describing the
maximal suborders of $\Op$ for every $\pp$.

\subsection*{Local Bass orders}

From here on we assume that $\pp\nmid 2$, and we fix $\delta\in\OO$ such that
$(\frac{\delta}{\pp})=-1$.

\medskip

The correspondence between isomorphism classes of Gorenstein orders in
quaternion algebras over local fields and ternary quadratic forms was
developed in \cite{brz-gor}. This correspondence was explored further
in \cite{joh}, where it is refined to describe Bass orders. We
summarize here the results we extract from this article.

\medskip

Let $\Op$ be an order, and let $\E=\{f_0,f_1,\-f_2,f_3\}$ be
a basis of $\Op^\vee$ satisfying
\begin{equation}\label{eqn:base-dual}
 \tr(f_0)=1,\quad\tr(f_1)=\tr(f_2)=\tr(f_3)=0.
\end{equation}

Denote by $M_{\E}$ the Gram matrix of the norm form
in the trace zero submodule of $\Op^\vee$ corresponding to $\E$, i.e.
\[
 M_{\E}=\big(\tr(f_i\bar f_j)\big)_{1\leq i,j\leq 3}.
\]
Then $d\cdot M_{\E}$ is the ternary quadratic form associated to
$\Op$, where $d$ is any generator of $d(\Op)$.

Conversely, to an integral ternary quadratic form $f$ over $\OO_\pp$
can be associated an order $C_0(f)$ in a quaternion algebra over
$K_\pp$: the order and the algebra are given by the even part of the
Clifford algebras associated to $f$ over $\OO_\pp$ and $K_\pp$
respectively.

By Propositions 5.8 and 5.10 of \cite{joh}, the maps $\Op\mapsto d\cdot M_\E$
and $f\mapsto C_0(f)$ give a bijection between isomorphism classes of
Bass orders in quaternion algebras over $K_\pp$ and the set of ternary
quadratic forms of Table \ref{table:qforms}, where we group forms into
\emph{classes} that will be treated in a unified way when convenient.

\begin{table}[h]
\scalebox{1}{
\begin{tabular}{|l|c|c|c|c|}
\hline
Class & Form & Parameters & Hilbert Symbol\\
\hline\hline
A1& $\< 1,-1,\pip^s >$ & $ s\geq 0$  & $1$\\
A2& $\< 1 , -\delta , \pip^s >$ & $s\geq 1$  &$(-1)^s$\\
B & $\< 1 , \pip , \epsilon_1 \pip>$ & $\epsilon_1\in\{1,\delta\}$ 
& $\kro{-\epsilon_1}{\pp}$\\
C &
$\<1,\epsilon_1\pip,\epsilon_2\pip^s>$&$\epsilon_1,\epsilon_2\in\{1,\delta\}, \,
s\geq2$ 
& $\kro{\epsilon_1}{\pp}^s\kro{-\epsilon_2}{\pp}$\\
\hline
\end{tabular}}
\caption{Ternary quadratic forms in correspondence with local
Bass orders.}\label{table:qforms}
\end{table}

In particular, every Bass order $\O$ in $B$ induces a family $(f_\pp)_\pp$ of
ternary quadratic forms, by letting $f_\pp$ be the form from Table
\ref{table:qforms} corresponding to $\Op$. This family satisfies that
$f_\pp=\< 1,-1,1 >$ for almost every $\pp$, and is
independent of the genus of $\O$.

Equation (\ref{eqn:norma}) below implies that, given a form $f=\<1,a,b >$, then
the quaternion algebra $C_0(f)\otimes_{\OO_\pp} K_\pp$ is a matrix algebra if
and only if $\< a,b,ab>$ is isotropic, i.e., if and only if the Hilbert symbol
$(\frac{-a,-b}{\pp})$ equals $1$. The sign for each case is shown in Table
\ref{table:qforms}.

The graphs on Figure \ref{fig:grafo} show how the isomorphism classes of Bass
orders  in quaternion algebras over $K_\pp$ are distributed. Each vertex
represents an isomorphism class of Bass orders, and there is an edge between two
vertices if and only if there is an order $\Op$ corresponding to the top vertex,
and an order $\Op'$ corresponding to the bottom vertex, such that $\Op'$ is a
maximal suborder of $\Op$; if $f$ and $g$ are respectively the corresponding
forms from Table \ref{table:qforms}, we will say that $g$ is \emph{beneath} $f$.
Note that these graphs reflect the assertion of Proposition \ref{prop:e=1_o_2}.

\begin{figure}[h]
\scalebox{1}{
\xymatrix@-1pc {
& & \text{Division} & &  & & & \text{Matrix} & & \\
  & & & & \OO_\pp & & & & & \bullet \ar@{-}[d] \ar@{-}[ddllll] \\
\bullet \ar@{-}[dd] \ar@{-}[drr]  & & & & \pp & & & & & \bullet
  \ar@{-}[d] \ar@{-}[dll] \\
 & & \bullet \ar@{-}[dl] \ar@{-}[dr] & &\pp^2 & \bullet \ar@{-}[dd] & &
  \bullet \ar@{-}[dl] \ar@{-}[dr] & &  \bullet \ar@{-}[d]\\
\bullet \ar@{-}[dd] & \bullet \ar@{-}[d] & &\bullet \ar@{-}[d] & \pp^3 &
& \bullet \ar@{-}[d] & & \bullet \ar@{-}[d] &  \bullet \ar@{-}[d]\\
&  \bullet \ar@{-}[d] & & \bullet \ar@{-}[d] & \pp^4  &\bullet \ar@{..}[dd] &
\bullet \ar@{-}[d] &
& \bullet \ar@{-}[d] &  \bullet
\ar@{-}[d]\\
\bullet \ar@{..}[d]&  \bullet \ar@{..}[d] & & \bullet \ar@{..}[d] & \pp^5 & 
\ar@{..}[d]  &
\bullet \ar@{..}[d] & & \bullet
\ar@{..}[d] & \bullet \ar@{..}[d]\\
A2  & C & B & C & & A2 & C & B & C & A1}
}
\caption{Isomorphism classes of Bass orders \label{fig:grafo}}
\end{figure}
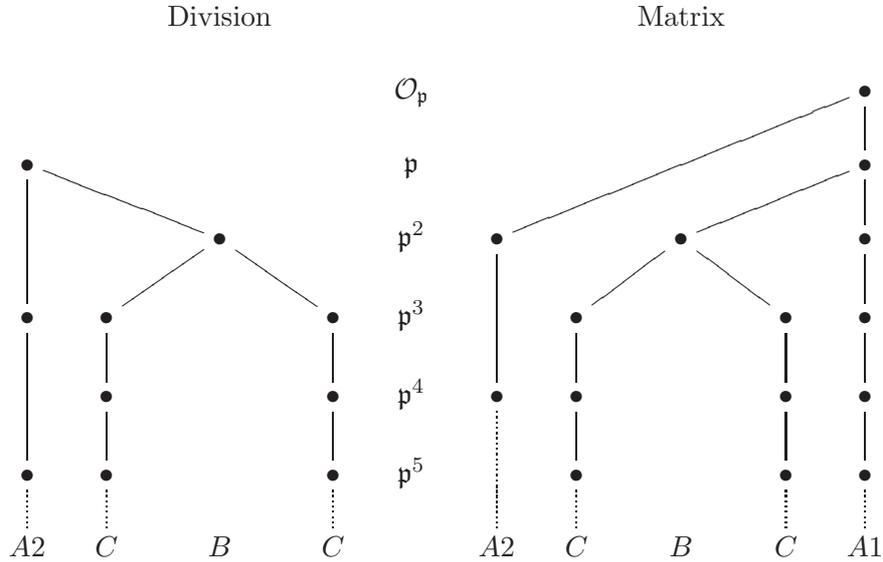

All the orders in the left graph lie in the division quaternion algebra, while
all the orders in the right one lie in the matrix algebra. Horizontally aligned
vertices have the same discriminant, which is indicated in the middle column.
Vertically aligned vertices correspond to forms of the same \emph{class}, which
is indicated in the bottom row. The orders of class A1 are the so called
\emph{Eichler orders} (see, e.g., Section 2 of \cite{brz-ord}), and the orders
of class A2 are the \emph{orders of level $p^{2r+1}$} considered in
\cite{piz-arithII}.

\begin{exm}
Let $s\geq0$. The order $E_s=\big(\begin{smallmatrix} \OO_\pp &
  \OO_\pp \\ \pip^s \OO_\pp & \OO_\pp
             \end{smallmatrix}
\big) \subseteq M_2(\OO_\pp)$ is a Bass order of class A1 and
discriminant $\pp^s$. $E_{s+1}$ is a maximal suborder of $E_s$.
\end{exm}

\begin{defi}
Let $\Op$ be a Bass order in correspondence with the form $f=\< 1 , a , b >$,
and let $\B=\{1,e_1,e_2,e_3\}$ be a basis of $\Op$ as an $\OO_\pp$-module. We
say that $\B$ is a \emph{good basis} if the $e_i$ satisfy
\begin{align}\label{eqn:tabla_mult}
 e_1^2&=-ab, & e_2^2&=-b, & e_3^2&=-a, \notag \\
 e_1 e_2&=-b e_3, & e_2 e_3 &= -e_1, & e_3 e_1 &= -a e_2,\\
e_2 e_1&=b e_3, & e_3 e_2 &= e_1, & e_1 e_3 &= a e_2. \notag
\end{align}

\end{defi}

Every Bass order has a good basis (see Section 4 of \cite{joh},
and also \cite{gross-luc}), and in such basis the norm form is given by

\begin{align}
 N &= \< 1 , ab , b , a>. \label{eqn:norma}
\end{align}

\begin{exm}

A good basis for the order $E_s$ defined above is given by
\[
1=\Big(\begin{array}{c c}
              1 & 0 \\
              0 & 1
             \end{array}
\Big),\quad e_1=\Big(\begin{array}{c c}
              0 & 1 \\
              \pip^s  & 0
             \end{array}
\Big),\quad e_2=\Big(\begin{array}{c c}
             0 & 1 \\
              -\pip^s & 0
             \end{array}
\Big),\quad e_3= \Big(\begin{array}{c c}
              1 & 0 \\
              0 & -1
             \end{array}
\Big).
\]
\end{exm}

\medskip

Let $\Op$ be an order in correspondence with the form $f=\< 1 , a , b >$, and
let $\E=\{f_0,f_1,f_2,f_3\}$ be a basis of $\Op^\vee$ satisfying
\eqref{eqn:base-dual}. Let $e_i=4ab\cdot f_j\bar{f_k}$, where $(i,j,k)$ is an
even permutation of $(1,2,3)$, and define $\E^\dagger=\{1,e_1,e_2,e_3\}$. Then
$\E^\dagger$ is a basis of $\Op$ (see \cite{joh}, Theorem 4.3).

\begin{prop}\label{prop:dualidad} 

With the notation as above, if $\E$ is such that
\begin{equation}\label{eqn:dual-qbuena}
 2ab \cdot M_{\E} = \diag(1,a,b),
\end{equation}
then $\E^\dagger$ is a good basis of $\Op$ (see \cite{joh}, Theorem 4.3).
\end{prop}

\begin{remark}\label{rmk:dualidad-recipr}
Conversely, if $\B$ is a good basis of $\Op$, then $M_{\B^\vee}$
satisfies (\ref{eqn:dual-qbuena}), where given a basis $\B=\{e_0,e_1,e_2,e_3\}$
of $\Op$, we denote by $\B^\vee=\{f_0,f_1,f_2,f_3\}$ the basis of $\Op^\vee$
characterized by the equations $\tr(e_i \bar f_j)=\delta_{ij}$.
\end{remark}

\subsection*{Constructing maximal suborders, the local case.}

Given an order $\Op$ correspon\-ding to a form $f$ from Table
\ref{table:qforms}, we construct a representative for each of the one or two
isomorphism classes of maximal suborders of $\Op$ (see Figure \ref{fig:grafo}).
The way to do this is, given a good basis $\{1,e_1,e_2,e_3\}$ of $\Op$ and a
form $g$ from Table \ref{table:qforms} beneath $f$, find elements
$d_1,d_2,d_3\in \Op$ satisfying the equations (\ref{eqn:tabla_mult})
corresponding to the form $g$. Then, the order $\Op'=\<
1,d_1,d_2,d_3>_{\OO_\pp}$ is a maximal suborder of $\Op$ in correspondence with
the form $g$, for which $\{1,d_1,d_2,d_3\}$ is a good basis.

\medskip

Using Hensel's Lemma, take $\alpha_0,\alpha_1,\beta_0,\beta_1,\mu,\nu\in
\OO_\pp$ satisfying:
\begin{itemize}
\item $\alpha_0^2-\alpha_1^2=\pip$.
\item $\beta_0^2+\beta_1^2=\delta$.
\item $\mu^2=-1$, when $(\frac{-1}{\pp})=1$.
\item $\nu^2=-\delta$, when $(\frac{-1}{\pp})=-1$.
\end{itemize}

\begin{prop} \label{prop:suborders}
The elements $d_1,d_2,d_3$ defined by Table \ref{table:suborders}
satisfy the equations \eqref{eqn:tabla_mult} corresponding to the form $g$.
\end{prop}
\begin{table}[h]
\scalebox{1}{
 \begin{tabular}{|l|l|l|l|}
\hline
Form & Form beneath & Good basis for $\Op'$\\
\hline\hline
$\< 1 , -1 ,
\pip^s > $ & $\< 1 ,
-1  , \pip^{s+1} > $ & $d_1=\alpha_0
e_1+\alpha_1 e_2,$\\

& &  $d_2= \alpha_1 e_1+\alpha_0 e_2, d_3=e_3$\\

$\< 1 , -1 ,
1 > $& $\< 1 , -\delta
,
\pip^2 > $& $d_1=\pip(\beta_1 e_1-\beta_0 e_3)$, \\

& & $d_2=\pip e_2, d_3=\beta_0 e_1 + \beta_1 e_3$\\

$\< 1 , -1 ,
\pip > $ & $\< 1 , \pip ,
\pip >$, if
$(\frac{-1}{\pp})=1$& $d_1=\mu  \pip e_3, d_2= \mu e_1, d_3=e_2$\\

& $\< 1 , \pip , \delta\pip
>$, if $(\frac{-1}{\pp})=-1$& $d_1=  \nu \pip e_3, d_2= \nu e_1,
d_3=e_2$\\

\hline

$\< 1 , -\delta , \pip^s >$
& $\< 1 , -\delta  , \pip^{s+2}
>$
&
$d_1=\pip e_1, d_2=\pip e_2, d_3=e_3$\\

$\< 1 , -\delta , \pip >$ &
$\< 1 , \pip , \delta \pip
>$, if $(\frac{-1}{\pp})=1$& $d_1=\mu \pip e_3, d_2= \mu e_1,
d_3=e_2$ \\

& $\< 1 , \pip , \pip >$,
if
$(\frac{-1}{\pp})=-1$& $d_1=\nu^{-1} \pip e_3, d_2= \nu^{-1} e_1,$\\

 & & $d_3=e_2$ \\

\hline

$\< 1 , \pip  , \pip >$ &
$\< 1 , \pip  , \pip^2 >$ &
$d_1= \pip e_2, d_2= e_1, d_3=e_3$ \\

& $\< 1 , \delta \pip
, \pip^2 >$ & $d_1=\pip(-\beta_1 e_2
+ \beta_0 e_3),$\\

& & $ d_2=e_1, d_3=\beta_0 e_2 + \beta_1 e_3$\\

 $\< 1 , \pip , \delta \pip
>$ & $\< 1 , \pip 
,  \delta \pip^2 >$ & $d_1= \pip e_2,
d_2= e_1,
d_3=e_3$ \\

& $\< 1 , \delta \pip
, \delta \pip^2 >$ & $d_1= \pip e_3,
d_2=e_1,
d_3=e_2$\\

\hline

$\< 1 , \pip ,  \pip^s> $ &
$\< 1 , \pip , \pip^{s+1}
>$ &
$d_1=
\pip e_2, d_2=e_1, d_3=e_3$ \\

$\< 1 , \delta \pip ,  \pip^s> $ & $\< 1
, \delta \pip , \delta \pip^{s+1} >$ & $d_1= \delta \pip e_2,
d_2=e_1, d_3=e_3$ \\

$\< 1 ,  \pip ,
\delta \pip^s> $ & $\< 1 , 
\pip , \delta \pip^{s+1} >$ & $d_1= \pip e_2, d_2=e_1, d_3=e_3$. \\

$\< 1 , \delta \pip , \delta \pip^s> $ & $\< 1
, \delta \pip ,\pip^{s+1}
>$ & $d_1= \delta\pip e_2, d_2= \delta^{-1} e_1,$\\
& & $d_3=e_3$ \\
\hline
\end{tabular}}
\caption{Suborders \label{table:suborders}}
\end{table}

\begin{proof}
In each case, it is straightforward to check that the $d_i$'s satisfy
the equations \eqref{eqn:tabla_mult} corresponding to $g$, using
that the $e_i$'s satisfy the equations corresponding to $f$.
\end{proof}

\begin{remark}\label{rmk:is_general}

It can be proved that this construction is general, in the sense that every
maximal suborder of $\Op$ can be obtained by the previous procedure, if we
start with a suitable good basis of $\Op$.

\end{remark}

\subsection*{Quasi-good bases}\label{subsection:qbuena}

So far, given an order $\Op$, we must obtain a good basis of it to compute
its suborders. This involves diagonalizing a ternary quadratic form over
$\OO_\pp$, which is not desirable from the computational viewpoint.
Nevertheless, as we will show in this subsection, this can be reduced to
diagonalize the corresponding form modulo $\pp^n$ for a certain small
non-negative integer $n$.

\begin{defi}\label{def:qbuena}
 
Let $\B=\{1,e_1,e_2,e_3\}$ be a basis of $\Op$. We say that $\B$ is a
\emph{quasi-good} basis if there exists a good basis $\tilde \B=\{1,\tilde
e_1,\tilde e_2,\tilde e_3\}$ of $\Op$ satisfying
\[
 \tilde e_i \equiv  e_i \mod(\pp\Op) \quad (1\leq i \leq 3).
\]

\end{defi}

\begin{prop}\label{prop:qbuena=>subord}

Let $\B=\{1,e_1,e_2,e_3\}$ be a quasi-good basis of an order
$\Op$ in correspondence with a form $f$, and let $g$ be a form beneath
$f$. Let $d_1,d_2,d_3$ be as in Table \ref{table:suborders}. Then,
\[
\Op'=\< 1, d_1, d_2, d_3>_{\OO_\pp}
\]
is a maximal suborder of $\Op$ in correspondence with the form $g$.

\end{prop}

\begin{proof}
Let $\tilde\B=\{1,\tilde e_1,\tilde e_2,\tilde e_3\}$ be a good basis of
$\Op$ as in Definition \ref{def:qbuena}. In terms of these elements and
the form $g$, define elements $\tilde d_1,\tilde d_2,\tilde d_3$
according to Table \ref{table:suborders}, and let $\Op'=\< 1,
\tilde d_1, \tilde d_2, \tilde d_3>_{\OO_\pp}$. The table shows
that $\tilde d_i \equiv d_i\mod (\pp \Op)$ for every $1\leq i\leq3$. Then,
since $\pp\Op\subseteq \Op'$, we have that
\[
\Op'=\Op'+\pp\Op=\< 1, \tilde d_1, \tilde d_2, \tilde
d_3>_{\OO_\pp}+\pp\Op \supseteq \< 1, d_1, d_2, d_3>_{\OO_\pp}.
\]
Letting $\Lambda_\pp$ denote the lattice on the right-hand side of this
equation, it suffices to see that $d(\Op')=d(\Lambda_\pp)$ to complete the
proof.

Let $e\in\{1,2\}$ be such that $[\Op:\Op']=\pp^e$. Following Table
\ref{table:suborders} case by case, it can be proved that $d(\Lambda_\pp)=\pp^e
d(\Op)$. Since $d(\Op')=\pp^e d(\Op)$, we are done.

\end{proof}

\begin{remark}\label{rmk:precision_param}
Let $m=v_\pp(d(\Op'))$. The proof shows that, when constructing the $d_i$'s, the
elements $\alpha_0,\alpha_1,\dots$ in Table \ref{table:suborders} need to be
calculated only up to precision $\pip^{m+1}$, since in that case the ideal
$d(\Lambda_\pp)$ remains unchanged.

It shows also that 
$\{1,d_1,d_2,d_3\}$ needs not to be a quasi-good
basis for $\Op'$, since we only get that $\tilde d_i \equiv d_i\mod (\pp \Op)$.
Nevertheless, since $\pp^2\Op\subseteq\pp\Op'$, it is a quasi-good basis if the
stronger congruence $\tilde e_i \equiv e_i \mod(\pp^2\Op)$ holds.

\end{remark}

Proposition \ref{prop:qbuena=>subord} shows that obtaining quasi-good bases is
enough for our purpose of computing suborders. In what follows we show how to
obtain these bases.

\medskip

Let $f=\<1,a,b>$ be the form in correspondence with $\Op$, and let
$\E=\{f_0,f_1,f_2,f_3\}$ be a basis of $\Op^\vee$ satisfying
(\ref{eqn:base-dual}). The existence of good bases implies that there exists
$C\in GL_3(\OO_\pp)$ such that $2ab\cdot C^t M_{\E} C = \diag(1,a,b)$.
Hence, $2ab\cdot M_{\E}\in M_3(\OO_\pp)$ and $\det(M_{\E})=8^{-1}(ab)^{-2}
u^2$ for some $u\in\OO_\pp^\times$.

\begin{prop}\label{prop:adj=>quasidist}
Let $n = 2v_\pp(a)+1$. Assume that $\E$ satisfies the following conditions.

\begin{enumerate}
\item There exists $\tilde b\in\OO_\pp$ such that
\[
 2ab\cdot M_{\E} \equiv \diag(1,a,\tilde b)
\mod (M_3(\pp^n \OO_\pp)).
\]
\item $\det(M_{\E})=8^{-1}(ab)^{-2} u^2$.
\end{enumerate}

Then, $\E^\dagger$ is a quasi-good basis of $\Op$.
\end{prop}

\begin{remark}\label{rmk:congr=>qbuena}

The congruence in \textit{(a)} is the really relevant hypothesis. If this
congruence is satisfied and $u\in\OO_\pp^\times$ is such that
$\det(M_{\E})=8^{-1}(ab)^{-2}u^2$, then the basis $\{f_0,f_1,f_2,u^{-1} f_3\}$
satisfies \textit{(a)} and also \textit{(b)}.

\end{remark}

The proof of Proposition \ref{prop:adj=>quasidist} is based on the following
lifting lemma.

\begin{lemma}\label{lemma:levantamiento}
Let $r, m$ be non negative integers such that $m>2r$, and let $A\in
M_3(\OO_\pp)$ be a symmetric matrix. Suppose that there exists $C\in
GL_3(\OO_\pp)$
such that
\[
 C^t A C \equiv \diag(\alpha,\beta,\gamma)
\mod (M_3(\pp^m \OO_\pp)),
\]
with $v_\pp(\alpha)=0$ and $v_\pp(\beta)=r$. Then, there exists $C'\in
GL_3(\OO_\pp)$
satisfying $C'\equiv C \mod (M_3(\pp^{m-r} \OO_\pp))$ such that
\[
 C'^t A C' \equiv \diag(\alpha',\beta',\gamma')
\mod (M_3(\pp^{m+1} \OO_\pp)),
\]
with $\alpha'\equiv\alpha\mod(\pp^{m-r}\OO_\pp)$ and
$\beta'\equiv\beta\mod(\pp^m\OO_\pp)$.
\label{lemma:diag}
\end{lemma}

\begin{proof}

Write
\[
 C^t A C = \diag(\alpha,\beta,\gamma)
+  \pip^m \Bigg( \begin{array}{ccc}
a & b & c\\
b & d & e\\
c & e & f
\end{array} \Bigg),
\]
with $a,b,\dots,f\in\OO_\pp$. We claim that there exists a matrix $C_0\in
GL_3(\OO_\pp)$ such that
\[
 C_0^t A C = \Bigg(\begin{array}{c c c}
              \alpha + a \pip^m & 0 & c'\pip^m\\
              -b\pip^r & \beta+d'\pip^m & e'\pip^m \\
          -c\pip^r & -e\pip^r & \gamma+f'\pip^m
             \end{array}
\Bigg),
\]
with $c',d',e',f'\in \OO_\pp$. This can be shown by performing row operations
on $C^t A C$, using the diagonal entries as pivots to first obtain zeroes at
the $(3,1),(2,1),(1,2)$ and $(3,2)$ entries, and then obtain
$-c\pip^r,-e\pip^r$ and $-b\pip^r$ at the $(3,1),(3,2)$ and $(2,1)$ entries
respectively.

Let $C'=C+\pip^{m-r}C_0$. Then,
\[
 C'^t A C' = \Bigg(\begin{array}{c c c}
              \alpha' & 0 & c'\pip^{2m-r}\\
              0 & \beta' & e'\pip^{2m-r} \\
          c'\pip^{2m-r} & e'\pip^{2m-r} & \gamma'
             \end{array}
\Bigg)+\pip^{2(m-r)}C_0^t A C_0.
\]
where $\alpha'=\alpha+a\pip^m + 2\pip^{m-r}(\alpha+a\pip^m)$ and
$\beta'=\beta+d'\pip^m + 2\pip^{m-r}(\beta+d'\pip^m)$. Since $2(m-r)\geq
m+1$, we are done.
\end{proof}

\begin{proof}[Proof of Proposition \ref{prop:adj=>quasidist}]

Let $r=v_\pp(a)$. By letting $m\to\infty$ in the previous lemma, we get a matrix
$C=(c_{ij})\in GL_3(\OO_\pp)$
satisfying $C\equiv I \mod(M_3(\pp^{2r+1} \OO_\pp))$ such that
\[
 2ab\cdot C^t  M_{\E} C =\diag(\alpha,\beta,\gamma),
\]
with $\alpha\equiv 1\mod(\pip^{r+1})$ and $\beta\equiv a\mod(\pip^{2r+1})$.
Using Hensel's lemma, take $x_1,x_2\in \OO_\pp^\times$ satisfying $x_i\equiv
1\mod(\pip^{r+1})$ such that $\alpha=x_1^2$ and $\beta= x_2^2 a$. Taking
determinants we see that $\gamma = x_3^2 b$, where $x_3=\frac{\det(C)}{x_1
x_2}$.

Now let $\tilde C=C\cdot\diag(x_1,x_2,x_3)^{-1}$. Then $\tilde C$ satisfies that
\[
 2ab\cdot {\tilde C}^t M_{\E} \tilde C =\diag(1,a,b).
\]
Let $\tilde f_i =\sum_{j=1}^3 \tilde c_{ji} f_j$, where $\tilde C=(\tilde
c_{ij})$, let $\tilde f_0=f_0$, and let
$\tilde\E=\{\tilde{f_0},\tilde{f_1},\tilde{f_2},\tilde{f_3}\}$. Then
$\tilde{\E}^\dagger$ is a good basis of $\Op$, for (\ref{eqn:dual-qbuena}) is
verified by $M_{\tilde\E}$. The congruences satisfied by the $x_i$'s and $C$
imply that $\tilde f_i \equiv f_i \mod (\pp \Op^\vee)$ for 
$1\leq i\leq3$. Hence $\E^\dagger$ is a quasi-good basis of $\Op$, since
Proposition 3.2 of \cite{brz-gor} gives that $4ab \cdot
\Op^\vee \Op^\vee \subseteq \Op$.

\end{proof}

\subsection*{From local to global}

Let $\Lambda$ be a lattice in $B$, and let
$\Lambda_\pp'\subseteq\Lambda_\pp$ be a sub\-lattice of index $\pp^e$, where $e$
is a non-negative integer. Let
$\Lambda' \subseteq B$ be the lattice given by
\[
 \Lambda_\qq'=\begin{cases}
        \Lambda_\qq & \text{ if }\qq\neq \pp, \\
        \Lambda_\pp' & \text{ if }\qq=\pp.
       \end{cases}
\]
Given a set of generators for $\Lambda$ as an $\OO$-module and a
set of generators for $\Lambda_\pp'$ as an $\OO_\pp$-module, how can
we construct a set of generators for $\Lambda'$ as an $\OO$-module?

Assume that $\Lambda=\< v_1,v_2,\dots,v_m >_{\OO}$
and that $\Lambda_\pp'=\< w_1,w_2,\dots,w_n
>_{\OO_\pp}$. For each $i$ write $w_i=\sum_j a_{ij} v_j$,
with $a_{ij}\in \OO_\pp$. There exist elements $b_{ij}\in \OO$ and
$c_{ij}\in \pip^e \OO_\pp$ such that $a_{ij}=b_{ij} + c_{ij}$ (they
can be constructed, for example, looking at the $\pp$-adic expansion of the
$a_{ij}$). Consider the vectors $\tilde w_i=\sum_j b_{ij}v_j$ (which belong to
$\Lambda$). Then we have

\begin{prop}\label{prop:latticeconstruction}
 \[
    \Lambda'= \pp^e \Lambda + \< \tilde w_1,\tilde w_2,\dots,\tilde
w_n
  >_{\OO}.
 \]
\end{prop}

\begin{proof}
It is enough to check that these two lattices coincide at all
localizations. Denote by $\Lambda''$ the lattice in the right hand side.
\begin{itemize}
\item If $\qq \neq \pp$, then $\pi_\pp$ is a unit in $\OO_\qq$. So $\pp^e
  \Lambda_\qq = \Lambda_\qq$, which implies that $\Lambda''_\qq= \Lambda_\qq+
\<
  \tilde w_1,\tilde w_2,\dots,\tilde w_n >_{\OO_\qq}=\Lambda_\qq$.

 \item Since $\pp^e\Lambda_\pp\subseteq\Lambda_\pp'$, we have that
$\Lambda''_\pp
\subseteq
   \Lambda_\pp$; the reverse inclusion is deduced from the fact that $\tilde w_i
   \equiv w_i\mod(\pp^e\Lambda_\pp)$.
\end{itemize}
\end{proof}

\begin{remark}\label{rmk:hnf}

Using the Hermite Normal Form algorithm (see \cite{cohen}, Chapter I), for every
lattice in $B$ we can compute a generating set over $\OO$ with at most five
elements. In particular, this can be done for the sum describing $\Lambda'$, and
we can assume that $\Lambda$ is given in this way.

\end{remark}

\subsection*{The algorithm}

We are now ready to prove our first main result, which we recall here.

\begin{thma}[]\label{thm:suborder}
 
There is an algorithm that, given a Bass order $\O$ in $B$, computes
suborders of $\O$ of any given genus.

\end{thma}

\begin{proof}

It suffices to give an algorithm which computes maximal suborders of $\O$ in any
given genus. So we assume that we are given a prime $\pp$, the form $f_\pp$
corresponding to $\Op$, and a form $g_\pp$ beneath $f_\pp$. The algorithm, which
we describe below, will return a Bass order $\O'\subseteq\O$ with $\Oq'=\Oq$ for
all $\qq\neq\pp$, and such that $\Op'$ corresponds to $g_\pp$.

\begin{algo}\label{algo:suborder}
\
\medskip

\noindent \textit{Step 1.} Use Proposition \ref{prop:adj=>quasidist} to find a
quasi-good basis for $\Op$.

\medskip

\noindent \textit{Step 2.} Use Proposition \ref{prop:qbuena=>subord} to
construct
a suborder $\Op'\subseteq\Op$ corresponding to the form $g_\pp$.

\medskip

\noindent \textit{Step 3.} Use Proposition \ref{prop:latticeconstruction} to
construct an
order $\O'$ such that
\[
 \O_\qq'=\begin{cases}
        \O_\qq & \text{ if }\qq\neq \pp, \\
        \O_\pp' & \text{ if }\qq=\pp.
       \end{cases}
\]
\end{algo}
\end{proof}

\section{Computing ideal classes representatives for suborders}

The aim of this section is to prove Theorem B. We start with some notation and
definitions.

If $\O$ is an order in $B$, we denote by $I(\O)$ the set
of left $\O$-ideals and by $\cl(\O)$ the set of equivalence
classes of left $\O$-ideals. The equivalence class of an ideal $I$ is
denoted by $[ I ]$. The \emph{norm} of an ideal $ I $ is defined as the
fractional ideal $N(I)\subseteq K$ generated by the elements $N(x)$ as $x$ runs
over
$ I $.

Throughout this section, let $\O'\subseteq\O$ be orders in $B$.

\begin{defi} If $ I \in I(\O)$, we define
\[
\ida=\{ J \in I(\O'):\O  J = I \},
\]
and we denote that set simply by $\Psi(I)$ when there is no possible confusion
on which are the orders under consideration.
\end{defi}

These sets will be considered for orders in $B$ as well as for their
completions. Both cases can and will be treated in an unified way.

\begin{remark}\label{rmk:psi_idelico}
 
When considering orders in $B$, identifying ideals with ideles, the set
$\Psi(I)$ is simply the preimage of $I$ under the natural map
\[
 {\widehat{\O'}}^\times \backslash {\widehat{B}}^\times \longrightarrow
{\widehat{\O}}^\times \backslash {\widehat{B}}^\times,
\]
where $\widehat{\phantom{x}}$ denotes tensor with $\widehat\ZZ$ over $\ZZ$.

\end{remark}

\begin{remark}\label{rmk:psi_PT}

The sets $\Psi(I)$ were studied in~\cite{villegas} to construct
modular forms of weight $2$ and level $p^2$ considering an order of
discriminant $p^2$, in the quaternion algebra over $\Q$ ramified at
$p$ and at $\infty$ - compare Corollary \ref{coro:ida=...} below and
equation (1) in \cite{villegas}.  They were later used in
\cite{tornaria} to construct modular forms of weight $3/2$ and level
$p^2$ considering in that algebra an order of class C at $p$.
\end{remark}

By $[\Psi(I)]$ we denote the set of classes of elements of $\Psi(I)$, i.e.
\[
[\Psi(I)]=\{[J]: J \in\Psi(I)\}.
\]
Note that if $[ I _1]=[ I _2]$, then
$[\Psi( I _1)]=[\Psi( I _2)]$.

\begin{prop}\label{prop:partic_clases_ideales}
\[
\cl(\O')=\coprod_{[I]\in \cl(\O)} [\Psi(I)].
\]
\end{prop}

\begin{proof}
This is straightforward using the idelic description of $\Psi(I)$, but we give a
direct proof.

Let $ J \in I(\O')$. Take $ I =\O J $. Then it is clear that
$ I \in I(\O)$ and $ J \in\Psi(I)$. This shows that the union on
the right hand side gives all of $\cl(\O')$.

It is clear that the union is disjoint, since if there are $J_i\in
\Psi_{\O'}^\O(I_i)$ for $i=1,2$ such that $[J_1]=[J_2]$,
then $[I_1]=[I_2]$. Indeed, let $x\in B^\times$ be such that
$J_1=J_2 x$. Then,
\[
 I_1=\O J_1=\O J_2 x=I_2 x.
\]
\end{proof}

This proposition shows that the sets $\Psi(I)$ can be used to give a system of
representatives for $\cl(\O')$, in terms of a system of representatives
for $\cl(\O)$. The next proposition shows that by constructing representatives
for $\cl(\O')$ using these sets, we will not enlarge the norms of the
$\O$-ideals that we start with.

\begin{prop}
Let $ J \in I(\O')$ such that $ J \subseteq I $. Then, $ J  \in
\Psi(I)$ if and only if $N(I) = N(J)$.
\end{prop}

\begin{proof}
Let $\qq$ be a prime of $\OO$. Since $J_\qq\subseteq I_\qq$ we can write
$ I _\qq= \Oq x_\qq$ and $ J _\qq = \O'_\qq z_\qq x_\qq$, with
$z_\qq\in\Oq$. Then, $N(I_\qq)= N(J_\qq)$ if and only if
$z_\qq\in\Oq^\times$, which is equivalent to the equality $\Oq J_\qq=I_\qq$.
These local facts imply the global statement.
\end{proof}

We have an action of the group $\Or(I)^\times$ on $\Psi(I)$ by right
multiplication, which stabilizes the left $\O'$-ideal classes.

\begin{prop}\label{prop:accion_unid_a_der}
If $ J  \in \Psi(I)$, then the action of $\Or(I)^\times$ on $[J] \cap \Psi(I)$
is transitive and the stabilizer of $ J $ is $\Or(J)^\times$. In
particular, $\# \big([J]\cap \Psi(I)\big)= [\Or(I)^\times:\Or(J)^\times]$.
\end{prop}

\begin{proof}
To prove that the action is transitive, let $ J_1 ,J_2\in \Psi(I)$ be
such that $[J_1]=[J_2]$. If $x\in B^\times$ is such that
$ J_1 =J_2 x$, then $x\in \Or(I)^\times$, since
$ I =\O J_1 =\O J_2 x= I x$. The other two statements are clear.
\end{proof}

The corollary below, which follows immediately, can be used to get
information about the class numbers, as we will see in Section
\ref{sect:ejemplo}.

\begin{coro}\label{coro:classnumber}
 \[
  \# \Psi(I) = \sum_{[J]\in\left[\Psi(I)\right]}[\Or(I)^\times:\Or(J)^\times].
 \]

\end{coro}

In what follows, we describe two different methods for calculating the set
$\Psi(I)$ for a given $I\in I(\O)$. The first one will rely on the action of the
units described above, in the local setting, whereas the second one will only
involve global calculations.

\subsection*{Local method: The action by
\texorpdfstring{$(\Op')^\times\backslash\Op^\times$}{the local units}}

\begin{prop}\label{prop:accion_local_unid_a_der}
 
Let $I_\pp\in I(\Op)$, say $I_\pp=\Op x_\pp$. Then, the map
\begin{align*}
(\Op')^\times\backslash \Op^\times & \longrightarrow \Psi(I_\pp) \\
\alpha_\pp & \mapsto \Op' (\alpha_\pp x_\pp)
\end{align*}
is bijective.

\end{prop}

\begin{proof}
 
This map is the composition of the maps
\begin{align*}
(\Op')^\times\backslash \Op^\times &\longrightarrow\Psi(\Op),       
&\Psi(\Op)&\longrightarrow \Psi(I_\pp). \\
\alpha_\pp & \mapsto \Op'\alpha_\pp     & J_\pp & \mapsto J_\pp x_\pp  
\end{align*}
Both maps are bijective. This is clear for the second map, and for the first one
this follows by Lemma \ref{prop:accion_unid_a_der}, since all $\Op$-ideals are
equivalent.

\end{proof}

\begin{prop}\label{prop:psi_local_global}
Suppose that  $[\O:\O']=\pp^e$ for some $e\geq1$. Let $ I \in I(\O)$. The map
\begin{align*}
\ida &\longrightarrow \Psi_{\Op'}^{\Op}( I_\pp) \\
 J &\mapsto J_\pp
\end{align*}
is bijective. In particular, $\#\ida=[\Op^\times : (\Op')^\times]$.
\end{prop}

\begin{proof}
The fact that $I_\qq = J_\qq$ for all $\qq \neq\pp$ implies that the map is
bijective. The equality follows from Corollary
\ref{prop:accion_local_unid_a_der}.
\end{proof}

These propositions imply immediately the following result.

\begin{coro}\label{coro:accion_local_unid_a_der}
 
Suppose that $[\O:\O']=\pp^e$ for some $e\geq1$. Let $I\in I(\O)$, and write
$I_\pp=\Op x_\pp$. If $\{\alpha_j\}$ is a system of representatives for
$(\Op')^\times \backslash \Op^\times$, then $\Psi^\O_{\O'}(I)=\{J_j\}$, where
$J_j$ is
locally given by
\[
( J _j)_\id{q} = \begin{cases}
 I _\id{q} & \text{ if } \id{q} \neq \id{p},\\
\Op' (\alpha_j x_\pp) & \text{ if } \id{q} = \id{p}.
\end{cases}
\]
\end{coro}

\begin{remark}\label{rmk:localgenerator}
A way to construct a local generator at $\pp$ of an ideal $ I $ is to
consider the entry with minimum valuation at $\pp$ of the Gram matrix of a
generating set  $\{w_1, \ldots, w_5\}$ for $ I $ over $\OO$, since the norm is
generated by an element with minimum valuation in such matrix. If this minimum
is attached in the entry $(i,j)$, then a local generator is $w_i+w_j$ if $i \neq
j$, and $w_i$ if $i=j$.
\end{remark}

\begin{prop}\label{prop:unid_mod_p} Assume that $\pp\Op\subseteq \Op'$. Then,
the natural map
\[
\phi: (\Op')^\times \backslash \Op^\times \longrightarrow 
(\pp\Op
\backslash \Op')^\times \backslash (\pp\Op \backslash \Op)^\times.
\]
is bijective.
\end{prop}

\begin{proof}
Consider the ring morphism $\phi_1:\Op\to\pp\Op\backslash\Op$. We claim that
the induced group homomorphism
$\phi_1:\Op^\times\to(\pp\Op\backslash\Op)^\times$ is surjective. Indeed, let
$[x]\in (\pp\Op\backslash\Op)^\times$. Then there exist $y,z\in\Op$ such that
$xy=1+\pip z$. Then $N(xy)\equiv 1 \mod(\pip)$, and hence $x\in\Op^\times$ as
claimed.

Compose $\phi_1$ with the map $p$ that projects
$(\pp\Op\backslash\Op)^\times$ onto the quotient set
$(\pp\Op \backslash \Op')^\times \backslash (\pp\Op
\backslash \Op)^\times$. Then $p\circ\phi_1$ is surjective, and
passes to the quotient set $(\Op')^\times \backslash \Op^\times$
to give a surjective map $\phi: (\Op')^\times \backslash \Op^\times \to
(\pp\Op \backslash \Op')^\times \backslash (\pp\Op
\backslash \Op)^\times$.

We claim that $\phi$ is injective. Indeed, let $x,y\in\Op^\times$ be
such that $\phi(x)=\phi(y)$. Then, since
$(\Op')^\times\to(\pp\Op\backslash\Op')^\times$ is also an epimorphism,
we have $z\in(\Op')^\times$ and $w\in\Op$ such that $x=zy+\pip
w$. Hence, $x=(z+\pip wy^{-1})y$, which shows that $[x]=[y]\in
(\Op')^\times \backslash \Op^\times$, since $\pip wy^{-1} \in \id{p} \Op
\subseteq\Op'$ and hence $z+\pip wy^{-1} \in (\Op')^\times$.
\end{proof}

By Proposition \ref{prop:e=1_o_2}, this result shows that, to give a system of
representatives for the sets $(\Op')^\times \backslash \Op^\times$ when $\Op'$
is a maximal suborder of $\Op$, it will be enough to do the calculations modulo
$\pp$.

Given a quasi-good basis $\B=\{1,e_1,e_2,e_3\}$ of $\Op$, and assuming that
$\Op'$ is obtained from $\Op$ by means of Algorithm \ref{algo:suborder}, we
proceed to give a system of representatives for the sets $(\Op')^\times
\backslash \Op^\times$, in terms of the form $g$ corresponding with $\Op'$. The
indexes $[\Op^\times :
(\Op')^\times]$ are known (see \cite{brz-aut}, Theorems 3.3 and 3.10), so it
will suffice to give in each case the right amount of non-equivalent units.

Let $q$ denote the order of the residue field $\kp$, and let
$\{a_1,a_2,\dots,a_q\}\subseteq\OO_\pp$ be a set of representatives for $\kp$
such that $a_1=1,a_2=-1$ and $a_q=0$. Let $\delta, \beta_0,\beta_1$ be as in
Proposition
\ref{prop:suborders}. Finally, let $S=\{
\tilde\gamma\in\kp\times\kp:1-\delta\tilde\gamma_1^2+\tilde\gamma_2^2\neq0\}$,
and each $\tilde\gamma\in S$ take $\gamma\in\OO_\pp\times\OO_\pp$ any lift of
$\tilde\gamma$.

\begin{prop}\label{prop:repr_unidades}

With the previous notation, Table \ref{table:repr_unidades} gives a system of
representatives for $(\Op')^\times\backslash\Op^\times$.

\medskip

\begin{table}[h]
\scalebox{0.76}{
 \begin{tabular}{|c|c|c|c|c|}
\hline
$\Op$-class & $\Op'$-class & $[\Op^\times:(\Op')^\times]$ & Representatives &
Condition\\
\hline\hline
\multirow{4}{*}{A1}  & \multirow{2}{*}{A1} & $q+1$ & $e_1, 1+\frac{a_i}{2}
(e_1 - e_2)\quad (1\leq i
\leq q)$ &
$d(\Op)=1$\\
\cline{3-5}
 &  & $q$ & $1 + \frac{a_i}{2} (e_1 - e_2)\quad (1\leq i \leq q)$ &
$d(\Op)\neq 1$\\
\cline{2-5}
 & A2 & $q(q-1)$ & $e_2, 1+\gamma_1 (\beta_1 e_3 - \beta_0 e_1) +
\gamma_2 e_2 \quad (\tilde\gamma\in S)$ &\\
\cline{2-5}
 & B & $q-1$ & $1,a_i+ e_3  \quad (3\leq i\leq q)$ & \\
\hline
\multirow{2}{*}{A2} & A2 & $q^2$ & $1+a_i e_1 + a_j e_2 \quad (1\leq
i,j\leq q)$&
\\
\cline{2-5}
 & B & $q+1$ & $1,a_i+ e_3 \quad (1\leq i\leq q)$& \\
\hline
\multirow{2}{*}{B} & \multirow{2}{*}{C} &\multirow{2}{*}{$q$} & $1,a_i+ e_2
\quad (1\leq i\leq q-1)$& $g\neq\<
1 , \delta\pip , \delta
  \pip^2 >$ \\
\cline{4-5}
 &  & &$1,a_i+ e_3 \quad (1\leq i\leq q-1)$& $g=\< 1
, \delta\pip ,
\delta \pip^2 >$\\
\hline
C & C & $q$ & $1,a_i+ e_2 \quad (1\leq i\leq q-1)$& \\
\hline
\end{tabular}}
\caption{Representatives for
$(\Op')^\times\backslash\Op^\times$}\label{table:repr_unidades}
\end{table}

\end{prop}

\begin{proof}

According to the Proposition \ref{prop:unid_mod_p} we may assume that $\B$ is a
good basis, and it suffices to calculate a
system of representatives for the set $(\pp\Op\backslash \Op')^\times
\backslash (\pp\Op \backslash \Op)^\times$.

First notice that $\pp\Op \backslash \Op$ is a $\kp$-algebra that inherits
naturally from $B_\pp$ a norm form $N:\pp\Op \backslash
\Op\to\kp$ such that $(\pp\Op \backslash \Op)^\times=\{x\in
\pp\Op \backslash \Op:N(x)\neq0\}$. This allows us to easily
check that all the given representatives are indeed units, and also to
give the needed description of $(\pp\Op \backslash \Op')^\times$.

We will do the details in a single case, namely when $\Op$ has class A1
and $\Op'$ has class B. The rest of the cases can be treated similarly.

Let $x=x_0+x_1 e_1+x_2 e_2+x_3 e_3 \in\pp\Op \backslash
\Op$. In these coordinates we have that the norm form is given by
$N(x)=x_0^2-x_3^2$ (see (\ref{eqn:norma})), and that $x\in\pp\Op \backslash
\Op'$ if and only if $x_3=0$.

Hence, the elements of the form $a_i+e_3$ belong
to $(\pp\Op \backslash \Op)^\times$, if $i\geq 3$. They are not equivalent
modulo $(\pp\Op \backslash \Op')^\times$, since if
\begin{multline*}
    (a_i+e_3)(x_0+x_1 e_1+x_2 e_2) = \\ =
  a_i x_0+(a_i x_1+x_2)
  e_1+(a_i x_2+x_1) e_2+x_0 e_3 =a_j +e_3 ,
\end{multline*}
then $x_0=1$ and hence $i=j$. And they are not equivalent to
$1$, since they do not belong to $\pp\Op \backslash \Op'$.
\end{proof}

\subsection*{Global method: The colon lattice} 

Let $ I \in I(\O)$. We introduce an alternative method to calculate $\Psi(I)$,
using global tools. Consider the lattice
\[
 \Lambda_ I =\{y\in B: y  I ^{-1}\subseteq \O'\}.
\]
It satisfies that $\Lambda_ I =\Lambda_\O  I$. For simplicity, we will just
consider $\Lambda =\Lambda_\O$. It is clear that $\Lambda\subseteq \O'$ and
$\O\subseteq R_r(\Lambda)$. 

\begin{lemma}\label{lemma:lambdaindex}
The lattice $\Lambda$ satisfies the following properties:

\begin{enumerate}
\item $\pp \O\subseteq \Lambda$, and hence $[\O:\Lambda]\mid \pp^4$.
\item $\Lambda\subseteq J$ for all $J\in\Psi(\O)$.
\end{enumerate}
\end{lemma}

\begin{proof}

The inclusion in \textit{(a)} follows from the fact that $\pp \O'\subseteq \O$.
The inclusion in \textit{(b)} is clear if we consider the completion at primes
$\qq\neq \pp$, so we will look only at the completion at $\pp$. Let  
$J\in\Psi(\O)$, and write $J_\pp=
  \Op' u_\pp$ with $u_\pp\in \Op^\times$. Then,
\[
 \alpha_\pp\in\Lambda_\pp\Rightarrow\alpha_\pp\Op\subseteq
 \Op'\Rightarrow \alpha_\pp u_\pp^{-1}\in \Op' \Rightarrow
 \alpha_\pp\in \Op' u_\pp = J_\pp.
\]

\end{proof}

Since $\pp \Op\subseteq \Op'$, we can consider $\Op/\Op'$ as a
$\kp$-vector space. When $e=2$, we can go further. Since in that case
$\Op'$ has class A2, the ring $\OO_\pp + \sqrt{\delta} \OO_\pp$ embeds into
$\Op'$, and hence into $\Op$. Then we can consider $\Op/\Op'$ as a
$\Kp$-vector space, where $\Kp$ is the quadratic extension of $\kp$
given by $\Kp=(\OO_\pp + \sqrt{\delta}\OO_\pp) / \pp(\OO_\pp +
\sqrt{\delta}\OO_\pp)$.

\begin{lemma}\
 \begin{enumerate}
\item If $e=1$, then $\dim_{\kp}(\Op/ \Op')=1$.
\item If $e=2$, then $\dim_{\Kp}(\Op/ \Op')=1$.
 \end{enumerate}
\end{lemma}

\begin{proof}

It follows immediately from the fact that $\vert \Op/\Op' \vert=q^e$.

\end{proof}

\begin{prop}\label{pp^2e}
$[\O':\Lambda]=\pp^e$, and hence $[\O:\Lambda]=\pp^{2e}$. In
  particular, if $e=2$ then $\Lambda=\pp R$.
\end{prop}

\begin{proof}
It is enough to consider the completion at $\pp$. Then, we need to show that
$\vert \Op' / \Lambda_\pp\vert=q^e$. Consider the morphism (of additive groups)
\begin{align*}
 \psi:\Op' &\rightarrow \End(\Op/\Op') \\
\alpha &\mapsto (v \mapsto \alpha \cdot v).
\end{align*}
Its kernel is $\Lambda_\pp$. The induced morphism $\psi:\Op'/\Lambda_\pp\to
\End(\Op/\Op')$ is easily seen to be also a $\kp$-vector space (respectively
$\Kp$-vector space) morphism when $e=1$ (respectively $e=2$). Note that since
$1\not \in \Lambda_\pp$, it is not the null morphism. Hence, the result follows
from the previous lemma.

\end{proof}

\begin{coro}\label{coro:ida=...} The set $\Psi(I)$ is given by
\[
 \Psi(I) = \{J:
\O J= I ,\,\Ol(J)=\O',\,\Lambda_{ I }\subseteq J \subseteq
 I ,\,[ I :J]=[J:\Lambda_{ I }]=\pp^e\}.
\]
\end{coro}

\begin{proof}

When $ I =\O$, the result follows immediately from Lemma \ref{lemma:lambdaindex}
and Proposition \ref{pp^2e}. The arguments used for the general case are
entirely ana\-logous.

\end{proof}

In particular, to calculate $\Psi(I)$ (whose cardinality we already know
by Proposition~\ref{prop:psi_local_global}), we can limit ourselves to
calculate the lattices between $\Lambda_{ I }$ and $ I $ with
the indicated indexes, and then determine which of them satisfy the first two
equalities. Furthermore, the equality $\Ol(J)=\O'$ can be replaced by
the equality $N(J) = N(I)$, which sometimes is easier to verify.

\begin{remark}
If $e=1$, then $[ I :\Lambda_{ I }]=\pp^2$, and there are $q+1$ lattices between
these two. We have seen that the number of elements of $\Psi(I)$ is $q-1,\, q$
or $q+1$. Hence, almost all lattices constructed are needed. This makes the
method effective. 
\end{remark}

\begin{remark}
In the case $e=2$, we know that the elements in $\Psi(I)$ have a $(\OO_\pp +
\sqrt{\delta} \OO_\pp)$-module structure. If we only consider lattices between
$\Lambda_{ I }$ and $ I $ which have this extra structure, there are $q^2+1$
such lattices. The order of $\Psi(I)$ is $q^2-q$ if $\O$ is the maximal order
and $\O'$ is of class A2, and $q^2$ if both orders are of class A2. Hence,
except for the maximal order, this construction is effective as well.
\end{remark}

\subsection*{The algorithm}

We now prove our second main result, which we first recall.

\begin{thmb}
\label{thm:representatives}
There is an algorithm that, given a Bass order $\O$ in $B$ and a set of
representatives $S$ of left $\O$-ideal classes, computes left ideal classes
representatives for suborders of $\O$ of any given genus. Furthermore, the set
of norms of the computed ideals is the same as the set of norms of the ideals in
$S$. 

\end{thmb}

\begin{proof} It suffices to give an algorithm that works when considering
maximal suborders of $\O$. In particular, we assume that we are given the same
input as in Algorithm \ref{algo:suborder}, plus the set $S$. The algorithm will
return a set $S'$ of representatives for left ideal classes representatives for
the suborder $\O'$ obtained by Algorithm \ref{algo:suborder}.
 
By Proposition \ref{prop:partic_clases_ideales}, it suffices to give an
algorithm which calculates, for each  $I\in S$, a set of representatives $S_I'$
for $[\Psi(I)]$, and then return $S'=\bigcup_{I\in S} S'_I$. The algorithm works
as follows.

\begin{algo}\label{algo:ideales}\
 
\medskip

\noindent\textit{Step 1}. Using Proposition \ref{prop:repr_unidades}, compute a
set of representatives for the set $(\Op')^\times \backslash \Op^\times$.

\medskip

\noindent\textit{Step 2}. Using Remark \ref{rmk:localgenerator}, 
find a local generator for $I_\pp$.

\medskip

\noindent\textit{Step 3}. Using Corollary \ref{coro:accion_local_unid_a_der} and
Proposition \ref{prop:latticeconstruction}, compute the set ${\Psi(I)}$.

\medskip

\noindent\textit{Step 4}. Set $T =\Psi(I)$ and set $S'_I=\emptyset$

\medskip

\noindent\textit{Step 4.1}. Pick $J\in T$ and compute $[J]\cap\Psi(I)$ by
letting $\O_r(J)^\times \backslash \Or(I)^\times$ act on $J$ (see Proposition
\ref{prop:accion_unid_a_der}).

\medskip

\noindent\textit{Step 4.2}. Set $S_I' = S_I' \cup \{J\}$. If $T \backslash
[J]=\emptyset$, return $S_I'$. Elseif, let $T=T \backslash [J]$, and go to Step
4.1.

\end{algo}
\end{proof}

\begin{remark}\label{rmk:local_vs_global}
We can replace Steps 1, 2 and 3 by the global method to compute $\Psi(I)$ given
in Corollary \ref{coro:ida=...}, although to our knowledge there is no advantage
of one over the other.
\end{remark}

We do not have a general method to, given $ J \in\Psi(I)$, compute a system of
representatives for the (finite) set $\Or(J)^\times \backslash \Or(I)^\times$
needed in Step 4.1. However, if $B$ is totally definite (i.e., if $K$ is totally
real and $B$ ramifies at every infinite place of $K$), then the set $\OO^\times
\backslash\Or(I)^\times $ is finite and can be used as well to compute
$[J]\cap\Psi(I)$. 

The finiteness of the set $\OO^\times \backslash\Or(I)^\times $, as well as a
method to compute it, can be obtained considering the exact sequence
\begin{equation}\label{eqn:s.e.c.}
  1 \longrightarrow\{\pm 1\}\backslash \Or(I)^{\times,1}
\longrightarrow \OO^\times \backslash \Or(I)^\times
\stackrel{N}{\longrightarrow}  (\OO^\times)^2\backslash \OO^\times_+,
\end{equation}
where $\OO^\times_+$ denotes the group of totally positive units of $\OO$.
Assuming $B$ totally definite, the quadratic form $\tr_{K\slash\Q}\circ
N:B\to\Q$ is positive definite, and hence the group $\Or(I)^{\times,1}$ is
finite and can be calculated using LLL. Furthermore, its possible group
structures are known (see \cite{vigneras-simpl}). The group
$(\OO^\times)^2\backslash \OO^\times_+$ is always finite, and equals the null
group in many cases, such as for fields $K$ having narrow class number equal to
$1$ (see \cite{edgar}).

\begin{remark}\label{rmk:iterando_ideales}
 
Since $\Or(J)^\times\subseteq \Or(I)^\times$ for every $J\in\Psi(I)$, when
iterating the algorithm we need to apply the previous procedure to compute the
sets $\OO^\times \backslash\Or(I)^\times $ only for the initial set of ideals.
\end{remark}

\section{Example: The Consani-Scholten quintic}\label{sect:ejemplo}

In this section we are going to show how we can use our method to compute ideal
classes representatives for an Eichler order of discriminant $30$ in the
quaternion algebra ramified at the two infinite places of the real quadratic
field $K=\QQ[\sqrt{5}]$.

This example was considered in \cite{consani} to conjecture the
modularity of a Galois representation attached to the third \'etale
cohomology of a quintic threefold (see \cite{consani}, Theorem 1 for
details). In that article the algebra considered is ramified also at
$2$ and $3$, since the Galois representation associated to the quintic
has semi-stable reduction at those places. The representatives are
constructed following the method of Pizer (see \cite{Pizer}), which
implies seeking for ideals and checking for equivalence between the
constructed ones until the class number (which has to be precomputed
or can be deduced during the computation using the Mass formula) is
reached. We consider instead the quaternion algebra ramified only at
the two infinite places, since in that case the maximal order has
class number equal to $1$, which makes calculations simpler. We make
use first of Theorem A to compute an Eichler order of discriminant $30$ and
then we make use of Theorem B to compute its left ideal class representatives.
Most of the computations were made with the aid of \texttt{SAGE} (\cite{sage}).

\medskip

Denote by $\omega = \frac{1+\sqrt{5}}{2}$ and let $\OO=\ZZ + \ZZ \omega$ be the
ring of integers of $K$. Let $B$ be the quaternion algebra $(-1,-1)_K$, i.e.,
the algebra over $K$ generated by $1,i,j,k$, where $i^2=j^2=-1,ij=k=-ji$. It is
clearly unramified at all finite places not dividing $2$, and it is ramified at
the two infinite places. Since $2$ is inert in the extension $K/\QQ$, $B$ does
not ramify at $2$ (by parity reasons).

\subsection{Constructing the orders}

Starting with a maximal order in $B$ as input, we compute an Eichler
order in $B$ of discriminant $30$. Considering the prime factorization of $30$ in
$\OO$, we iterate Algorithm \ref{algo:suborder} to construct a chain
\[
R(1)\supseteq R(2)\supseteq R(3)\supseteq R(6)\supseteq
R(6\sqrt{5})\supseteq R(30),
\]
where $R(\mathfrak{N})$ denotes an order of discriminant $\mathfrak{N}$.

\medskip

The maximal order we use is the order given in Chapter V of \cite{vig}, namely
\[
R(1) = \<\dfrac{1+\omega^{-1}i+\omega j}{2}, \dfrac{\omega^{-1} i + j + \omega
k}{2}, \dfrac{\omega i+ \omega^{-1} j +k}{2},\dfrac{ i + \omega j + \omega^{-1}
k}{2} >_\OO.
\]

\subsubsection{Discriminant 2}

In this first step we use Algorithm \ref{algo:suborder} referring
to the Appendix.

\medskip

\noindent\textit{Step 1}. The order $R(1)_2$ is in correspondence with the
form $f=H \perp \< 1 >$. Using the basis for $R(1)$ given above, we get that
\[
 \B=\big\{1,\tfrac{1}{2}(1+\omega^{-1}i+\omega j),\tfrac{1}{2}(\omega
i+\omega^{-1}j+k),\tfrac{1}{2}(i + \omega j + \omega^{-1} k)\big\}
\]
is a basis for $R(1)_2$. Its dual basis is
\begin{align*}
  \B^\vee=\big\{&f_0,\omega i -(1+\omega)k,\tfrac{1}{2}\big((1+\omega)i-j-\omega
k\big),\\ 
&\tfrac{1}{2}\big(-(1+2\omega)i+\omega j+(1+3\omega)k\big)\big\},
\end{align*}
where $f_0=\frac{1}{2}(1-\omega i+(1+\omega)k)$. Diagonalizing $M_{\B^\vee}$ (as
a ternary quadratic form), we see that letting
\begin{align*}
 f_1&=\tfrac{1}{5}\big((2+\omega)i-j-(1+\omega)k\big), \\
 f_2&=\tfrac{1}{2}\big((1+\omega)i-j+(6+11\omega)k\big), \\
 f_3&=\tfrac{1}{5}\big(-(47+88\omega)i+(11+26\omega)j+(43+32\omega)k\big),
\end{align*}
the hypotheses of Proposition \ref{prop:adj=>quasidist_H} are satisfied by
$\E=\{f_0,f_1,f_2,f_3\}$. Hence, letting
\begin{align*}
 e_1&=\tfrac{1}{2}
\big(-(232+384\omega)-(79+119\omega)i-(265+212\omega)j-(2-5\omega)k\big), \\
e_2
&=\tfrac{1}{25}\big(268+444\omega+(6-31\omega)i-(17+84\omega)j-(1+\omega)k\big)
,\\
e_3 &=\tfrac{1}{10}\big(13+24\omega-(7+12\omega)i-(10+21\omega)j)-k\big),
\end{align*}
we get that $\E^\dagger=\{1,e_1,e_2,e_3\} $ is a quasi-good basis for
$R(1)_2$.

\medskip

\noindent\textit{Step 2}. We are descending from $f=H \perp \< 1 >$ to $g=H
\perp \< 2 >$. To illustrate Remark \ref{rmk:is_general}, we show that we can
construct a well-known order of discriminant $2$. For this purpose,
we conjugate the quasi-good basis found above by $x=e_1+e_2$ (which belongs to
$R(1)_2^\times$, by Table \ref{table:repr_unidades2}), thus obtaining another
quasi-good basis of $R(1)_2$. Proposition \ref{prop:suborders2} gives then that
$\{1,x e_1 x^{-1},2\cdot
xe_2x^{-1},xe_3x^{-1}\}$ is a basis of $R(2)_2$.

\medskip

\noindent\textit{Step 3}. Applying Proposition \ref{prop:latticeconstruction} to
this basis, we obtain that
\[
R(2) = \<1,i,j,\frac{1+i+j+k}{2}>_\OO
\]
is an Eicher order of discriminant 2. Note that the given basis is a basis for
the classical maximal order in the quaternion algebra $(-1,-1)_\Q$.

\subsubsection{Discriminant 6} Diagonalizing modulo
$3$ the quadratic form associated to $\{x\in {R(2)_3}^\vee:\tr(x)=0\}$,
we obtain using Proposition \ref{prop:adj=>quasidist} that
$\{1,\frac{1}{2}(i+j),\frac{k}{2},2(i-\nobreak j)\}$ is a quasi-good basis for
$R(2)_3$.

We use Table \ref{table:suborders} to descend from $\< 1,-1 ,1> $ to $\<
1,-1 ,3>$, using $\alpha_0=2,\alpha_1=-1$ as parameters, and we get that a
basis for $R(6)_3$ is given by $\{1,i+j-\frac{k}{2},-\frac{1}{2}
(i+j)+k,2(i-j)\}$. Using Proposition \ref{prop:latticeconstruction}, we get that
\[
R(6) = \<1, i+2k, 3k, \frac{1+i+j+k}{2}>_\OO
\]
is an Eichler order of discriminant 6.

\subsubsection{Discriminant $6\sqrt{5}$} The basis
$\E=\{\frac{1}{2},-i,-\frac{k}{2},-\frac{j}{4}\}$ of
${R(6)_{\sqrt{5}}}^\vee$ satisfies the hypotheses of Proposition
\ref{prop:adj=>quasidist}, but with a stronger congruence in
\textit{(a)}, namely mod $(\sqrt{5})^2$. This implies that the basis
for $R(6\sqrt{5})_{\sqrt{5}}$ obtained below is a quasi-good basis
(see Remark \ref{rmk:precision_param}).

We apply Table \ref{table:suborders} using
$\alpha_0=2+\frac{\omega}{3},\alpha_1=-2$ as parameters, and obtain
$\{1,-(1+\frac{\omega}{6})i+2k,i-(2+\frac{\omega}{3})k,-2j\}$ as a
basis for $R(6\sqrt{5})_{\sqrt{5}}$. Then Proposition
\ref{prop:latticeconstruction} gives
\[
R(6\sqrt{5}) = \<1, i+2k, 3\sqrt{5}k, \frac{1+i+j+7k}{2}>_\OO.
\]

\subsubsection{Discriminant $30$} To construct $R(30)$, we use the quasi-good basis
obtained in the previous step and
$\alpha_0=\frac{139}{82}+\frac{61}{123}\omega,\alpha_1=-2$ as
parameters. The basis for $R(30)_{\sqrt{5}}$ obtained in this way is
$\{1,-(\frac{34}{9}+\frac{31}{36}\omega)i+(\frac{303}{41}+\frac{68}{41}\omega)k,
(\frac{303}{82}+\frac{34}{41}\omega)i-(\frac{68}{9}+\frac{31}{18}\omega)k,-2j\}
$. Applying Proposition \ref{prop:latticeconstruction}, we obtain
\[
R(30) = \<1, i+2k, 15k, \frac{1+i+j+7k}{2}>_\OO.
\]

\subsection{Constructing the ideals}

We now proceed to compute ideal classes representatives for $R(30)$ iterating
Algorithm \ref{algo:ideales}, and using the quasi-good bases obtained above.

\medskip

Before starting, note that Equation (\ref{eqn:s.e.c.}) implies that only norm
one global units need to be considered when checking for equivalence of ideals
in Step $\mathit{4.1}$, since $K$ has narrow class number $1$.

\medskip

In \cite{vig} it is shown that $R(1)$ has class number equal to one. It is also
shown that $R(1)^{\times,1}=E_{120}$, where $E_{120}$ is the binary icosahedral
group. Using this explicit description we can avoid the use of LLL.
Furthermore, by Remark \ref{rmk:iterando_ideales}, this group contains all of
the global units needed in our computations.

\subsubsection{Discriminant 2} The calculation of $\cl(R(2))$ can be done
without using the algorithm. Since
$\vert R(2)^{\times,1} \vert =24$ and $[R(1)_2^\times:R(2)_2^\times]=5$ (see
Table \ref{table:repr_unidades2}), Corollary \ref{coro:classnumber} implies that
$\left[\Psi_{R(2)}^{R(1)}(R(1))\right]=[R(2)]$, from which we conclude that
$R(2)$ has class number equal to $1$ as well.

\subsubsection{Discriminant 6} We now compute $\cl(R(6))$, following closely
Algorithm \ref{algo:ideales}. We have $S=\{R(2)\}$ as input.

\medskip

\noindent\textit{Step 1}. To obtain a set of representatives for
$R(6)_3^\times \backslash R(2)_3^\times$, we use
$\{0, 1, 2, \omega, \break 2\omega, \omega+1, \omega+2, 2\omega+1, 2\omega+2\}$ as a
set of representatives for $k_3$.

\medskip

\noindent\textit{Step 2}. The ideal $R(2)_3$ is trivially generated by $1$,
so there is no need to use Remark \ref{rmk:localgenerator} in this case.

\medskip

\noindent\textit{Steps 3 and 4}. The set $\Psi_{R(6)}^{R(2)}(R(2))$ has ten
ideals, which we do not list for length reasons. The action of
$R(2)^{\times,1}$ on $\Psi_{R(6)}^{R(2)}(R(2))$ has two orbits, namely $[I]$
and $[J]$, where $I=R(6)$ and $J$ is the $R(6)$-ideal corresponding to the
fifth generator of $R(6)_3^\times \backslash R(2)_3^\times$, which is given by
\[
 J = \< i+(\omega-1)k, j-(\omega+1)k, 3k,
1+\dfrac{\omega}{2}\big(3-i-j-3k\big)>_\OO.
\]
This result agrees with Corollary \ref{coro:classnumber}, since $\vert
\Or(I)^{\times,1}
\vert =6,\,\vert \Or(J)^{\times,1}\vert =4$ and
$[R(2)_3^\times:R(6)_3^\times]=10$.

\medskip

Hence, the algorithm gives that $\cl(R(6))=\{[I],[J]\}$.

\subsubsection{Discriminant $6\sqrt{5}$}

We compute $\cl(R(6\sqrt{5}))$ in the same way as before. We avoid
writing down all the details but give enough information so the reader
may verify the computations easily.
\begin{itemize}
\item We take $\{0,1,2,3,4\}$ as a set of
representatives for $k_{\sqrt{5}}$.
\item $1$ is a local generator of $J_{\sqrt{5}}$, since
$J_{\sqrt{5}}=R(6)_{\sqrt{5}}$.

\item Denote
$\Psi_{R(6\sqrt{5})}^{R(6)}(I)=\{I_1,\dots I_6 \}$ and
$\Psi_{R(6\sqrt{5})}^{R(6)}(J)=\{J_1,\dots J_6 \}$, where the notation
is such that the $n$-th ideal corresponds to the $n$-th representative
of $R(6\sqrt{5})_{\sqrt{5}}^\times \backslash
R(6)_{\sqrt{5}}^\times$.

\item The action of $\Or(I)^{\times,1}$ on $\Psi_{R(6\sqrt{5})}^{R(6)}(I)$ gives
that $
\left[\Psi_{R(6\sqrt{5})}^{R(6)}(I)\right]
=\left\{[I_1],[I_4]\right\}$, and the action of $\Or(J)^{\times,1}$ on $
\Psi_{R(6\sqrt{5})}^{R(6)}(J)$ gives that
  $\left[\Psi_{R(6\sqrt{5})}^{R(6)}(J)\right]
  =\left\{[J_1],[J_2],[J_3],[J_5] \right\}$ (see Table
  \ref{table:cl(R(6u))} for an explicit description of these ideals).

\end{itemize}

Hence, $\cl(R(6\sqrt{5}))=\{[I_1],[I_4],[J_1],[J_2],[J_3],[J_5]\}$.
This agrees with Corollary \ref{coro:classnumber}, since we have that $\,\vert
\Or(I_1)^{\times,1}\vert=\vert\Or(I_4)^{\times,1}\vert=\vert\Or(J_1)^{
\times,1}\vert=\vert\Or(J_3)^{\times,1}\vert=2$, and $\vert
\Or(J_2)^{\times,1}\vert =\vert \Or(J_5)^{\times,1}\vert =4$.

\begin{table}[h]
\scalebox{.88}{
 \begin{tabular}{|l|l|l|}
\hline
  Ideal & Basis & Ideal above\\
\hline\hline
 $I_1$ & $i+2k,3\sqrt{5}k, 1,\frac{1}{2}(1+i+j+7k)$ & \multirow{2}{*}{$I$}\\
$I_4$&$i+2k,3\sqrt{5}k,j+14k,\frac{1}{2}(1+i+j+19k)$&\\
\hline
$J_1$&$i+(\omega-1)k,3\sqrt{5}k,j-(\omega+7)k,\frac
{1}{2}(1-i-j+(18+\sqrt{5}
)k)$  & \multirow{4}{*}{$J$}\\
$J_2$&$i+(\omega-1)k,3\sqrt{5}k,j-(\omega+4)k,\frac{1}{2}
(1-i-j+(6+\sqrt{5})k)
$&\\
$J_3$&$i+(\omega-1)k,3\sqrt{5}k,j-(\omega+1)k,\frac{1}{2}
(1-i+j+(6-\sqrt{5})k)$&\\
$J_5$ & $i+(\omega-1)k,3\sqrt{5}k,j-(\omega-5)k,\frac{1}{2}(1-i-j+\sqrt{5}k)$&\\
\hline
\end{tabular}}
\caption{Representatives for $\cl(R(6\sqrt{5}))$ \label{table:cl(R(6u))}}
\end{table}

\subsubsection{Discriminant $30$}
Finally, we compute $\cl(R(30))$. 
\begin{itemize}
\item The residue field is the same as before, so we take the same
  representatives for $k_{\sqrt{5}}$.
\item The local generators at $\sqrt{5}$ for the ideals in
  $\cl(R(6\sqrt{5}))$ were constructed using Corollary
  \ref{coro:accion_local_unid_a_der}. They are $1,1 -\frac{3}{4}i +
\frac{3}{2}k  ,  1  ,1 - \frac{i}{4} + \frac{k}{2},1 -
\frac{i}{2} + k$ and $1 - i + 2k$ for $I_1,I_4,J_1,J_2,J_3$ and $J_5$
respectively.

\item Since
$\Or(I_1)^{\times,1}=\Or(I_4)^{\times,1}=\Or(J_1)^{\times,1}=\Or(J_3)^{
\times,1}=\{\pm1\}$, between the ideals belonging to
$\Psi^{R(6\sqrt{5})}_{R(30)}(I_1)$,$\Psi^{R(6\sqrt{5})}_{R(30)}(I_4),\Psi^{
R(6\sqrt{5})}_{R(30)}(J_1)$ and $\Psi^{R(6\sqrt{5})}_{R(30)}(J_3)$ there are no
equivalences. 
\item The action of $\Or(J_2)^{\times,1}$ on $\Psi^{R(6\sqrt{5})}_{R(30)}(J_2)$
gives that $
\left[\Psi^{R(6\sqrt{5})}_{R(30)}(J_2)\right]=\left\{[J_{2,1}],[J_{2,2}],[J_{2,3
}]\right\}$, and the action of $\Or(J_5)^{\times,1}$ on
$\Psi^{R(6\sqrt{5})}_{R(30)}(J_5)$ gives that
$\left[\Psi^{R(6\sqrt{5})}_{R(30)}(J_5)\right]=\left\{[J_{5,1}],[J_{5,2}],[J_ {
5,3}] \right\}$ (see Table \ref{table:cl(R(30))}).
\end{itemize}

In particular, $\#\cl(R(30))=4\cdot 5 + 6 = 26.$

\begin{table}[h]
\scalebox{.89}{
 \begin{tabular}{|l|l|l|}
\hline
  Ideal & Basis & Ideal above\\
\hline\hline
$I_{1,1}$ &$i+2k,15k,1,\frac{1}{2}(1+i+j+7k)$&\multirow{5}{*}{$I_1$}\\
$I_{1,2}$ & $i+2k, 15k,j+2(1+3\omega)k, ,\frac{1}{2}(1+i+j+(7-6\sqrt{5})k) $
&
\\
$I_{1,3}$ & $ i+2k,15k, j-(1+3\omega)k, \frac{1}{2}(1+i+j+(-8+3\sqrt{5})k)$
&
\\
$I_{1,4}$ & $i+2k,15k, j-(4-3\omega)k, \frac{1}{2}(1+i+j+(8+3\sqrt{5})k)  $
& \\
$I_{1,5}$ & $i+2k,15k, j-(7+6\omega)k, \frac{1}{2}(1+i+j+(7+6\sqrt{5})k) $
& \\
\hline
$I_{4,1}$ & $i+2k,15k,j+2(2-3\omega)k,\frac{1}{2}(1+i+j-(11+6\sqrt{5})k) $ &
\multirow{5}{*}{$I_4$}\\
$I_{4,2}$ & $i+2k,15k,j-(7+3\omega)k,\frac{1}{2}(1+i+j+(4+3\sqrt{5})k) $ &
\\
$I_{4,3}$ & $i+2k,15k,j+(5+3\omega)k,\frac{1}{2}(1+i+j+(1-6\sqrt{5})k)$ & \\

$I_{4,4}$ & $i+2k,15k,j+2(1-3\omega)k,\frac{1}{2}(1+i+j+(19+6\sqrt{5})k) $ &
\\
$I_{4,5}$ & $i+2k,15k,j+14k,\frac{1}{2}(1+i+j+19k) $ & \\
\hline
$J_{1,1}$ &
$i+(2-5\omega)k,15k,j+5(1+\omega)k,\frac{1}{2}(1+i+j+(2-5\sqrt{5})k) $
&\multirow{5}{*}{$J_1$} \\
$J_{1,2}$ &
$i+(2-5\omega)k,15k,j+(2-4\omega)k,\frac{1}{2}(1+i+j+(17+4\sqrt{5})k) $ &
\\
$J_{1,3}$ &
$i+(2-5\omega)k,15k,j+(2-\omega)k,\frac{1}{2}(1+i+j-(13+2\sqrt{5})k)
$ & \\
$J_{1,4}$ &
$i+(2-5\omega)k,15k,j-(4+7\omega)k,\frac{1}{2}(1+i+j+(2+7\sqrt{5})k)
$ & \\
$J_{1,5}$ &
$i+(2-5\omega)k,15k,j-(1+7\omega)k,\frac{1}{2}(1+i+j+(2+\sqrt{5})k)
$ & \\
\hline
$J_{2,1} $ &
$i+(2-5\omega)k,15k,j+(5-7\omega)k,\frac{1}{2}(1+i+j-(4+5\sqrt{5})k)
$ & \multirow{3}{*}{$J_2$}\\
$J_{2,2} $ &
$i+(2-5\omega)k,15k,j+(5-4\omega)k,\frac{1}{2}(1+i+j+(11+4\sqrt{5})k) $ & \\
$J_{2,3}$ &
$i+(2-5\omega)k,15k,j+2(1+\omega)k,\frac{1}{2}(1+i+j+(11-2\sqrt{5})k)
$ & \\
\hline
$J_{3,1}$ &
$i+(2-5\omega)k,15k,j-(4-5\omega)k,\frac{1}{2}(1+i+j-(10+5\sqrt{5})k) $ &
\multirow{5}{*}{$J_3$}\\

$J_{3,2}$ &
$i+(2-5\omega)k,15k,j-(7+4\omega)k,\frac{1}{2}(1-i+j+(5+4\sqrt{5})k)
$ & \\
$J_{3,3}$ &
$i+(2-5\omega)k,15k,j+(5+2\omega)k,\frac{1}{2}(1+i+j+(5-2\sqrt{5})k) $ & \\

$J_{3,4}$ &
$i+(2-5\omega)k,15k,j+(2-7\omega)k,\frac{1}{2}(1+i+j+(20+7\sqrt{5})k) $ &
\\
$J_{3,5}$ &
$i+(2-5\omega)k,15k,j+(1+\omega)k,\frac{1}{2}(1+i+j-(10-\sqrt{5})k) $ &
\\

\hline
$J_{5,1} $
&$i+(2-5\omega)k,15k,j+(2+5\omega)k,\frac{1}{2}(1+i+j+(8-5\sqrt{5})k) $ &
\multirow{3}{*}{$J_5$}\\
$J_{5,2} $
&$i+(2-5\omega)k,15k,j-(1+4\omega)k,\frac{1}{2}(1+i+j+(15+4\sqrt{5})k) $ &
\\
$J_{5,3}$ &
$i+(2-5\omega)k,15k,j-2(2-\omega)k,\frac{1}{2}(1-i+j-(6-3\sqrt{5})k)
$ & \\
\hline
\end{tabular}}
\caption{Representatives for $\cl(R(30))$ \label{table:cl(R(30))}}
\end{table}

\medskip

We end this section remarking that all the results obtained agree with
Eichler's mass formula (\cite{vig}, Corollaire V.2.3).

\section{Appendix: The case
\texorpdfstring{$\pp=(2)$}{p=(2)}}\label{section:p=2}

Assume that $\pp=(2)$, i.e. that $2$ is inert in $K\slash\Q$. The only
difference with the case $\pp\nmid2$ lays in the ternary quadratic forms that we
need to consider to describe Bass orders. We will study these forms in this
section.

Consider the matrices
\[
 H= \left(\begin{array}{c c}
              0 & 1\\
              1 & 0
             \end{array}
\right),\qquad
J= \left(\begin{array}{c c}
              2 & 1\\
              1 & 2
             \end{array}
\right),
\]
and given $f,g$ quadratic forms, let $f\perp g$ denote their orthogonal
sum. According to Propositions 5.8 and 5.12 of \cite{joh}, isomorphism
classes of Bass orders in quaternion algebras over $K_2$ are in one to one
correspondence with the forms $f$ of Table \ref{table:qforms2}. As in the case
$\pp\nmid2$, orders of class A1 are the so called \emph{Eichler orders}.

\begin{table}[h]
\scalebox{1}{
\begin{tabular}{|l|c|c|c|r|}
\hline
Class & Form & Parameters & Condition & Algebra\\
\hline\hline
A1& $H\perp \< 2^s >$ & $ s\geq 0$ & & $1$\\
A2& $J \perp \< 2^s >$ & $s\geq 1$ & & $(-1)^s$\\
B& $\< 1 , 1
  , \delta_1 2^s >$ & $s\geq
0,\delta_1\in\{1,3\}$ & $\delta_1=1$ & $-1$\\
& & &  $\delta_1=3$ & $1$\\
C & $\< 1 , 6 , \delta_1 2^s
>$ & $s\geq 1,\delta_1\in\{1,3\}$ &
$\delta_1 = 1$  & $(-1)^s$\\
& & & $\delta_1 = 3$ & $(-1)^{s+1}$ \\
D & $\< 1 , 5 , \delta_1 2^s
>$ & $s\geq 3,\delta_1\in\{1,3\}$ &
$\delta_1 = 1$   & $(-1)^{s+1}$\\
& & & $\delta_1 = 3$ & $(-1)^s$ \\
E & $\< 1 , 2 , \delta_2 2^s
>$ & $s\geq 3,\delta_2\in\{1,5\}$ &
$\delta_2=1$ & $-1$\\
& & & $\delta_2 = 5$ & $1$ \\
F & $\< 1 , 14 , \delta_2 2^s
>$ & $s\geq 4,\delta_2\in\{1,5\}$ &
$\delta_2=1$ & $1$\\
& & & $\delta_2 = 5$ & $-1$ \\
G & $\< 1 , 10 , \delta_2 2^s
>$ & $s\geq 4,\delta_2\in\{1,5\}$ &
$\delta_2=1$ & $(-1)^{s+1}$\\
& & & $\delta_2 = 5$ & $(-1)^s$ \\
\hline
\end{tabular}}
\caption{Ternary quadratic forms, when $\pp=(2)$.\label{table:qforms2}}
\end{table}

On the right column of Table \ref{table:qforms2} we indicate
with $1$ or $-1$ whether the order $C_0(f)$ belongs to the matrix
algebra or to the division algebra. As before, this depends on whether
the norm form associated to $C_0(f)$ is isotropic or not. We omit the
calculations.

Figure
\ref{grafo2} shows how isomorphism classes of Bass orders in quaternion
algebras over $K_2$ are distributed.

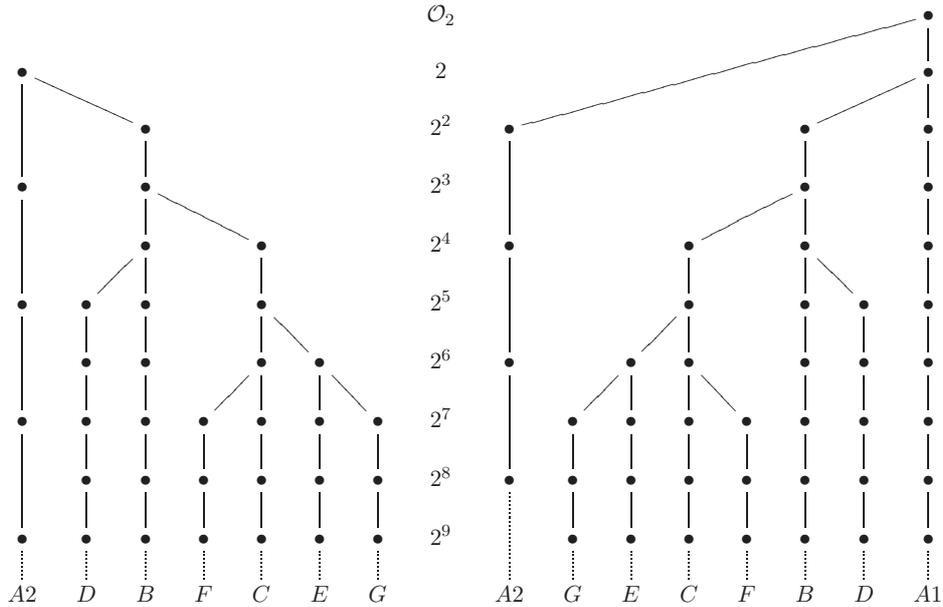
\begin{figure}[h]
\scalebox{0.8}{
\xymatrix@-1pc {
 & & & & & & & \OO_2 & & & & & & & & \bullet \ar@{-}[d] \ar@{-}[ddlllllll]\\
\bullet \ar@{-}[dd] \ar@{-}[drr]  & & & & & & &  2 & & & & & & & & \bullet
\ar@{-}[d] \ar@{-}[dll] \\
 & & \bullet \ar@{-}[d] & & & & &  2^2  & \bullet \ar@{-}[dd]& & & & & \bullet
\ar@{-}[d] & & \bullet \ar@{-}[d] \\
\bullet \ar@{-}[dd]   & & \bullet \ar@{-}[d] \ar@{-}[drr] & &  & &  & 2^3 & & &
& & & \bullet \ar@{-}[d]  \ar@{-}[dll] & & \bullet \ar@{-}[d]\\
 & & \bullet \ar@{-}[d] \ar@{-}[dl] & & \bullet \ar@{-}[d]  & &  & 2^4 & \bullet
\ar@{-}[dd] & & & \bullet \ar@{-}[d] & & \bullet \ar@{-}[d]  \ar@{-}[dr] & &
\bullet \ar@{-}[d]\\
\bullet \ar@{-}[dd]& \bullet \ar@{-}[d] & \bullet  \ar@{-}[d] & &  \bullet
\ar@{-}[d] \ar@{-}[dr] & & & 2^5 & & & & \bullet \ar@{-}[d] \ar@{-}[dl] &
&\bullet \ar@{-}[d] & \bullet \ar@{-}[d]& \bullet \ar@{-}[d]\\
 &  \bullet \ar@{-}[d] & \bullet \ar@{-}[d] & &  \bullet \ar@{-}[d] \ar@{-}[dl]
&   \bullet \ar@{-}[d] \ar@{-}[dr] & & 2^6  & \bullet \ar@{-}[dd] & &
\bullet\ar@{-}[d]\ar@{-}[dl] & \bullet \ar@{-}[d]\ar@{-}[dr]& & \bullet
\ar@{-}[d]& \bullet \ar@{-}[d]& \bullet \ar@{-}[d]\\
\bullet \ar@{-}[dd]& \bullet \ar@{-}[d] & \bullet \ar@{-}[d] & \bullet
\ar@{-}[d]& \bullet \ar@{-}[d] & \bullet \ar@{-}[d]& \bullet\ar@{-}[d] & 2^7 & &
\bullet \ar@{-}[d]& \bullet \ar@{-}[d]&\bullet \ar@{-}[d] & \bullet \ar@{-}[d]&
\bullet \ar@{-}[d]& \bullet \ar@{-}[d]&\bullet \ar@{-}[d]\\
 & \bullet \ar@{-}[d] & \bullet \ar@{-}[d] & \bullet \ar@{-}[d]& \bullet
\ar@{-}[d] & \bullet \ar@{-}[d]& \bullet\ar@{-}[d] & 2^8 & \bullet \ar@{..}[dd]&
\bullet \ar@{-}[d]&\bullet \ar@{-}[d] & \bullet \ar@{-}[d]&\bullet \ar@{-}[d] &
\bullet \ar@{-}[d]& \bullet \ar@{-}[d]&\bullet \ar@{-}[d]\\
\bullet \ar@{..}[d] & \bullet \ar@{..}[d] & \bullet \ar@{..}[d] & \bullet
\ar@{..}[d]& \bullet \ar@{..}[d] & \bullet \ar@{..}[d]& \bullet\ar@{..}[d] & 2^9
& & \bullet \ar@{..}[d]& \bullet \ar@{..}[d]& \bullet \ar@{..}[d]& \bullet
\ar@{..}[d]& \bullet \ar@{..}[d]&\bullet \ar@{..}[d] &\bullet \ar@{..}[d]\\
A2 & D & B & F & C & E & G & & A2 & G & E & C & F & B & D & A1}
}
\caption{Isomorphism classes of Bass orders, when $\pp=(2)$. \label{grafo2}}
\end{figure}

The notion of good basis must be extended to include the
non-diagonal forms of Table \ref{table:qforms2}.

\begin{defi}
Let $\O_2$ be a Bass order in correspondence with the form $f=H \perp
\< 2^s >$ (respectively, $f=J \perp \< 2^s >$). A basis $\B=\{1,e_1,e_2,e_3\}$
of $\O_2$ as an $\OO_2$-module is \emph{good} if the $e_i$ satisfy
\begin{align}\label{eqn:tabla_mult_H}
 e_1^2&=0, & e_1e_2&=2^s(1-e_3), & e_2 e_1&=2^s e_3, \notag \\
 e_2^2&=0, & e_2 e_3 &=0, & e_3 e_2 &= e_2,\\
e_3^2&=e_3, & e_3 e_1&=0, & e_1 e_3&=e_1. \notag
\end{align}

Respectively, if the $e_i$ satisfy
\begin{align}\label{eqn:tabla_mult_J}
 e_1^2&=-2^s, & e_1e_2&=2^s(1-e_3), & e_2 e_1&=2^s e_3, \notag \\
 e_2^2&=-2^s, & e_2 e_3 &=-e_1, & e_3 e_2 &= e_1+e_2,\\
e_3^2&=e_3-1, & e_3 e_1&=-e_2, & e_1 e_3&=e_1+e_2. \notag
\end{align}
\end{defi}

Note that in such bases the norm form is given by
\begin{equation}\label{eqn:norma2}
  N(x)=\begin{cases}
    x_0^2+x_0 x_3 -2^s x_1 x_2, & f=H \perp\< 2^s >,\\
x_0^2+x_0 x_3 +x_3^2 -2^s x_1 x_2 +2^s x_1^2 + 2^s x_2^2, & f=J
\perp\< 2^s >.
   \end{cases}
\end{equation}

\begin{remark}
 
We can extend Remark \ref{rmk:dualidad-recipr} to non-diagonal forms as follows.
Let $\O_2$ be an order in correspondence with $f=H \perp \< 2^s >$, and let $\B$
be a good basis of $\O_2$. Then,
\[
 -2^s \cdot M_{\B^\vee} = \left(\begin{array}{c c c}
              0 & 1 & 0\\
              1 & 0 & 0 \\
          0 & 0 & 2^{s+1}
             \end{array}
\right).
\]
Respectively if $\O_2$ is in correspondence with $f=J \perp \< 2^s >$, then
\[
 2^s 3 \cdot M_{\B^\vee} = \left(\begin{array}{c c c}
              2 & 1 & 0\\
              1 & 2 & 0 \\
          0 & 0 & 2^{s+1}
             \end{array}
\right).
\]
\end{remark}

\medskip

In order to state the analogue of Proposition \ref{prop:suborders}, using
Hensel's lemma take $\mu_1,\dots,\mu_8\in \OO_2$ satisfying:
\begin{align*}
&\bullet \mu_1^2 =-7 &  &\bullet 3\mu_2^2=-13\\
&\bullet 25\mu_3^2=1 & & \bullet 9 \mu_4^2 =1\\
&\bullet 3 \mu_5^2 =-5 & & \bullet \mu_6^2=-15\\
&\bullet 3\mu_7^2=-29  & &\bullet 3\mu_8^2=-533.
\end{align*}

\begin{prop}\label{prop:suborders2}

Let $\O_2$ be an order corresponding to a form $f$ from Table
\ref{table:qforms2}, and let $\{1,e_1,e_2,e_3\}$ be a
good basis for $\O_2$. Let $g$ be a form beneath $f$, and let
$d_1,d_2,d_3$ be as in Table \ref{table:suborders2}.

Then, $\O_2'=\< 1,d_1,d_2,d_3>_{\OO_2}$ is a
maximal suborder of $\O_2$ in correspondence with the form $g$, of
which $\{1,d_1,d_2,d_3\}$ is a good basis.
\end{prop}

\begin{table}[h]
\scalebox{1}{
\begin{tabular}{|l|l|l|l|}
\hline
Form & Form beneath & Good basis for $\Op'$\\
\hline\hline

$H \perp \< 1 >$ & $J \perp\< 4
>$ & $d_1=2(\mu_1-2 e_1 -3 e_2-2\mu_1e_3),$\\
& & $d_2=2(-\mu + 3e_1+2 e_2 +2 \mu_1 e_3),$\\
& & $d_3= -2 -\mu_1 e_1\mu_1 e_2+5e_3$\\

$H \perp \<2^s > $ & $H \perp \<2^{s+1}
> $ & $d_1=e_1, d_2= 2 e_2, d_3=e_3$\\

$H \perp \< 2 >$ & $\< 1
, 1 , 3
>$ & $d_1=\mu_1-e_1+2e_2-2\mu_1 e_3,$\\
& &  $d_2= -5+2\mu_1 e_1 +\mu_1 e_2+10e_3,$\\
& & $d_3=\mu_1+3 e_1 + e_2 -2\mu_1e_3$\\

$J \perp \< 2^s > $& $ J \perp \<
2^{s+2} > $& $d_1= 2 e_1, d_2=2 e_2, d_3=e_3$\\

$J \perp \< 2 >$ & $\< 1
, 1 , 1
>$ & $d_1=\mu_2-e_1+2e_2-2\mu_2 e_3,$\\
& & $d_2= \mu_2- 2 e_1 + e_2 -2 \mu_2 e_3,$\\
& & $d_3=-3-\mu_2 e_1+ \mu_2 e_2 + 6 e_3$\\

 $\< 1 , 1 ,
\delta_1  2^s >$ & $\< 1 , 1
, \delta_1  2^{s+1} >$ &
$d_1=e_1-e_2,d_2=e_1+e_2,d_3=e_3$\\

 $\< 1 , 2 ,
\delta_2 2^s >$ & $\< 1 , 2
, \delta_3  2^{s+1} >$ &
$d_1=-2e_2,d_2=e_1,d_3=e_3$\\

 $\< 1 , 5 ,
2^s >$ & $\< 1 , 5
, 3 \cdot2^{s+1} >$ &
$d_1=e_1-5e_2,d_2=e_1+e_2,d_3=e_3$\\

 $\< 1 , 6 ,
2^s >$ & $\< 1 , 6
, 3 \cdot 2^{s+1} >$ &
$d_1=-6e_2,d_2=e_1,d_3=e_3$\\

 $\< 1 , 10
, 2^s >$ & $\< 1
, 10 , 5 \cdot2^{s+1}
>$ & $d_1=-10e_2,d_2=e_1,d_3=e_3$\\

 $\< 1 , 1 ,
6 >$ & $\< 1 , 6
, 6 >$ & $d_1=6 e_3,d_2=e_2,d_3=-e_1$\\

 $\< 1 , 1 ,
2 >$ & $\< 1 , 6
, 2 >$ & $d_1=2e_1+6
e_3,d_2=e_2,d_3=2e_3-e_1$\\

 $\< 1 , 1 ,
2^2 >$ & $\< 1 , 5
, 3 \cdot 2^3 >$ & $d_1=e_1-5 e_2+4
e_3,d_2=e_1+e_2+4 e_3,$\\& & $d_3=e_3-e_1$\\

 $\< 1 , 14
, \delta_2 2^s >$ & $\< 1
, 14 , \delta_2
2^{s+1} >$ & $d_1=e_1-14\mu_1 e_2,d_2=\mu_1 e_1+e_2,d_3=e_3$\\

 $\< 1 , 5 ,
3 \cdot2^s >$ & $\< 1 , 5
,  2^{s+1} >$ & $d_1=e_1 - 5 \mu_2 e_2
,d_2=\mu_2 e_1+e_2,d_3=e_3$\\

 $\< 1 , 10
, 5 \cdot 2^s >$ & $\< 1
, 10 ,  2^{s+1}
>$ & $d_1=-10\mu_3 e_2,d_2=\mu_3 e_1,d_3=e_3$\\

 $\< 1 , 6 ,
3 \cdot 2^s >$ & $\< 1 , 6
,   2^{s+1} >$ & $d_1=2 e_1-6\mu_4
e_2,d_2=\mu_4 e_1 + 2 e_2,$\\& &$d_2=\mu_4 e_1 + 2 e_2,d_3=e_3$\\

 $\< 1 , 6 ,
3 \cdot 2^2 >$ & $\< 1 , 2
, 2^3 >$ & $d_1=6\mu_4(-\mu_5 e_2+2
e_3),d_2=\mu_4 e_1,$\\& & $d_3=e_2 +\mu_5 e_3$\\

 $\< 1 , 2 ,
2^3 >$ & $\< 1 , 1
, 2^4 >$ & $d_1=-2 e_2+ 8 e_3,d_2=
e_1,d_3=e_2 +5 e_3$ \\

 $\< 1 , 2 ,
5 \cdot 2^3 >$ & $\< 1 , 10
, 5 \cdot 2^4 >$ & $d_1=-2 \mu_6 e_2+
40 e_3,d_2= e_1,$\\& & $d_3=e_2 +\mu_6 e_3$ \\

 $\< 1 , 6 ,
3 \cdot 2^3 >$ & $\< 1 , 14
, 2^4 >$ & $d_1=6 \mu_4(-\mu_5 e_2+4
e_3),d_2=\mu_4 e_1,$\\& & $d_3=e_2 +\mu_5 e_3$ \\

$\< 1 , 6
, 2^2 >$ & $\< 1
, 2 , 5\cdot
2^3 >$ & $d_1=2(\mu_7 e_1 - 3\mu_7 e_2 -10
e_3),$\\
& & $d_2=e_1+2e_2,d_3=e_1 -3 e_2+\mu_7 e_3$ \\

$\< 1 , 6
, 2^3 >$ & $\< 1
, 14 ,
5\cdot 2^4 >$ & $d_1=2(\mu_7 e_1 - 3\mu_7 e_2 -60
e_3),$\\
& & $d_2=e_1+2e_2,d_3=3 e_1 -9 e_2+\mu_7 e_3$ \\ \hline
\end{tabular}}
\caption{Suborders \label{table:suborders2}}
\end{table}

\begin{proof}

All the cases can be easily checked. Many of them follow from Propositions
\ref{prop:general_1}, \ref{prop:general_2} and \ref{prop:general_3} below (see
the proof of Proposition \ref{prop:repr_unidades2}).

\end{proof}

The notion of quasi-good basis remains unchanged, as well as the use of such
bases for computing suborders and representatives for the quotients
$(\O_2')^\times\backslash\O_2^\times$. We must show how to quasi-good bases in
the $2$-adic case.

\begin{remark}\label{rmk:hensel2}
 
Proposition \ref{prop:adj=>quasidist} still holds for diagonal forms, setting
$n=3v_2(a)+2$ in order to be able to use Hensel's lemma in its proof.

\end{remark}

\begin{prop}\label{prop:adj=>quasidist_H}
Let $\O_2$ be an order in correspondence with $f=H \perp \< 2^s >$. Let
$\E=\{f_0,f_1,f_2,f_3\}$ be a basis of $\O_2^\vee$ satisfying
(\ref{eqn:base-dual}). Assume that $\E$ satisfies the following
conditions.

\begin{enumerate}
 \item There exists $\beta\in\OO_2$ such that
\[
 -2^s \cdot M_{\E} \equiv \left(\begin{array}{c c c}
              0 & 1 & 0\\
              1 & 0 & 0 \\
          0 & 0 & \beta
             \end{array}
\right) \mod (M_3(2^3 \OO_2)).
\]
\item $\det(M_{\E})=2^{1-2s}$.
\end{enumerate}

Let $e_i=-2^s\cdot f_j\bar{f_k}$, where $(i,j,k)$ is an even
permutation of $(1,2,3)$. Then, $\E^\dagger=\{1,e_1,e_2,e_3\}$ is a quasi-good
basis of $\O_2$.

\end{prop}

The following lifting lemma is needed in the proof of Proposition
\ref{prop:adj=>quasidist_H}, which is quite similar to the proof of Proposition
\ref{prop:adj=>quasidist}, and we omit.

\begin{lemma}\label{lemma:levantamientoH}
Let $m$ be an integer such that $m\geq3$, and let $A\in
M_3(\OO_2)$ be a symmetric matrix. Assume that there exists $C\in GL_3(\OO_2)$
such that
\[
 C^t A C \equiv \Bigg(\begin{array}{c c c}
              0 & \alpha & 0\\
              \alpha & 0 & 0 \\
          0 & 0 & \beta
             \end{array}
\Bigg) \mod (M_3(2^m \OO_2)),
\]
with $v_2(\alpha)=0$. Then, there exists $C'\in GL_3(\OO_2)$
satisfying $C'\equiv C \mod (M_3(2^{m-1} \OO_2))$ such that
\[
 C'^t A C' \equiv \Bigg(\begin{array}{c c c}
              0 & \alpha' & 0\\
              \alpha' & 0 & 0 \\
          0 & 0 & \beta'
             \end{array}
\Bigg) \mod (M_3(2^{m+1} \OO_2)),
\]
with $\alpha'\equiv\alpha\mod(2^{m-1}\OO_2)$.
\end{lemma}

\begin{proof}

Write
\[
 C^t A C = \Bigg(\begin{array}{c c c}
              0 & \alpha & 0\\
              \alpha & 0 & 0 \\
          0 & 0 & \beta
             \end{array}
\Bigg)  +  2^m \Bigg( \begin{array}{ccc}
a & b & c\\
b & d & e\\
c & e & f
\end{array} \Bigg),
\]
with $a,b,\dots,f\in\OO_2$. We claim that there exists a matrix
$C_0\in GL_3(\OO_2)$ such that
\[
 C_0^t A C = \Bigg(\begin{array}{c c c}
               -a  & b' & c'2^m\\
              d' & -d & e'2^m \\
          -2c & -2e & f'
             \end{array}
\Bigg),
\]
with $b',c',d',e',f'\in \OO_2$. This can be shown by
performing row operations on $C^t A C$, using the $(1,2)$ and $(2,1)$ entries as
pivots to first obtain zeroes at the $(1,1),(2,2),(3,1)$ and $(3,2)$ entries,
and then obtain $-a,-d,-2c$ and $-2e$ at the $(1,1),(2,2),(3,1)$ and
$(3,2)$ entries respectively.

Now let $C'=C+2^{m-1}C_0$. Then,
\[
 C'^t A C' = \Bigg(\begin{array}{c c c}
              0 & \alpha' & c'2^{2m-1}\\
              \alpha' & 0 & e'2^{2m-1} \\
          c'2^{2m-1} & e'2^{2m-1} & \beta'
             \end{array}
\Bigg)+2^{2(m-1)}C_0^t A C_0.
\]
where $\alpha'=\alpha+2^{m-1}(b'+d')$. Since $2(m-1)\geq m+1$, we
are done.
\end{proof}

For orders of class A2 we only state the corresponding analogue of Proposition
\ref{prop:adj=>quasidist_H}.

\begin{prop}\label{prop:adj=>quasidist_J}
Let $\O_2$ be an order in correspondence with $f=J \perp \< 2^s >$. Let
$\E=\{f_0,f_1,f_2,f_3\}$ be a basis of $\O_2^\vee$ satisfying
(\ref{eqn:base-dual}). Assume that $\E$ satisfies the following conditions.

\begin{enumerate}
 \item There exists $\beta\in\OO_2$ such that
\[
 2^s 3 \cdot M_{\E} \equiv \left(\begin{array}{c c c}
              0 & 1 & 0\\
              1 & 0 & 0 \\
          0 & 0 & \beta
             \end{array}
\right) \mod (M_3(2^3 \OO_2)).
\]
\item $\det(M_{\E})=2^{1-2s}3^{-2}$.
\end{enumerate}

Let $e_i=2^s 3\cdot f_j\bar{f_k}$, where $(i,j,k)$ is an even permutation of
$(1,2,3)$. Then, $\E^\dagger=\{1,e_1,e_2,e_3\}$ is a quasi-good basis of $\O_2$.
 
\end{prop}

\medskip

Finally, we proceed to give systems of representatives for the quotient sets
$(\O_2')^\times\backslash\O_2^\times$ when $\O_2'$ is a maximal suborder of
$\O_2$ obtained using Algorithm \ref{algo:suborder}. We start stating three
general results which, though stated and used only when $\pp=(2)$, hold without
restrictions on $\pp$.

Let $\B=\{1,e_1,e_2,e_3\}$ be a good basis for $\O_2$. Let $q$ be the
order of the residue field $k_2$, and let
$a_1,a_2,\dots,a_q\in\OO_2$ be a set of representatives for $k_2$.

\begin{prop}\label{prop:general_1}

Suppose that $\O_2$ is in correspondence with the form $f=\< 1 , a , b >$, and
let $\lambda\in \OO_2$. Assume that there exist $\alpha_0,\alpha_3\in \OO_2$
such that $\alpha_0^2+a\alpha_3^2=\lambda$. Let $v=\alpha_0+\alpha_3 e_3$, and
let $d_1= v e_1, d_2= v e_2, d_3=e_3$.

Then, $\O_2'=\< 1,d_1,d_2,d_3>_{\OO_2}$ is a suborder of $\O_2$ in
correspondence with the form $g=\< 1 , a , \lambda b >$, of which
$\{1,d_1,d_2,d_3\}$ is a good basis. Furthermore, if $v_2(\lambda)=1$ and
$v_2(b)\geq1$, then $\O_2'$ is a maximal suborder of $\O_2$, the index of
$(\O_2')^\times$ in $\O_2^\times$ is $q$, and a set of representatives for the
set $(\O_2')^\times\backslash\O_2^\times$ is given by $\{1+ a_i e_2 :1\leq i
\leq q\}$.
\end{prop}

\begin{proof}
The first assertion is easily checked. To prove the second one, we use
Proposition \ref{prop:unid_mod_p}. Since $v_2(b)\geq1$, by
(\ref{eqn:norma}) the norm form on $2 \O_2\backslash\O_2$ is given by
$N(x)=x_0^2+a x_3^2$. Hence, $\vert(2
\O_2\backslash \O_2)^\times\vert=c\cdot q^2$, where $c=\#\{(x_0,x_3)\in
k_2:x_0^2+a x_3^2\neq 0\}$.

We have that
\[
 2 \O_2\backslash\O_2'=\big\{x\in  2 \O_2\backslash\O_2:x_0,x_3\in
k_2,(x_1,x_2)\in A(k_2)\big\},
\]
where $A\in\End_{ k_2}( k_2\times k_2)$ is the morphism given by left
multiplication by $\big(\begin{smallmatrix} \alpha_0 & \alpha_3
  \\ -\alpha_3 & \alpha_0
             \end{smallmatrix}
\big)$. Since $\alpha_0^2+a \alpha_3^2=\lambda$ and
$v_2(\lambda)=1$, this matrix has rank $1$. Hence,
$\vert(2 \O_2\backslash \O_2')^\times\vert=c\cdot q$, which shows that
$[\O_2^\times:(\O_2')^\times]=q$.

To see that the given units are not equivalent, take $x\in
(2 \O_2\backslash\O_2')^\times$. Then, it is easy to see that
\begin{align*}
 (1+a_i e_2)x &=x_0+(x_1-a_i x_2 x_3)e_1+(a_i x_0+x_2)e_2 + x_3 e_3 \\
&= 1+a_j e_2
\end{align*}
implies that $i=j$.

\end{proof}

The next two results can be proved following the same ideas as the ones used
above.

\begin{prop}\label{prop:general_2}
Suppose that $\O_2$ is in correspondence with the form $f=\< 1 , a , b >$, and
let $\mu\in \OO_2$. Assume that there exist $\alpha_0,\alpha_2\in \OO_2$ such
that $\alpha_0^2+b\alpha_2^2=\mu$. Let $v=\alpha_0+\alpha_2 e_2$, and let $d_1=
v e_1, d_2= e_2, d_3=v e_3$.

Then, $\O_2'=\< 1,d_1,d_2,d_3>_{\OO_2}$ is a suborder of $\O_2$ in
correspondence with the form $g=\< 1 , \mu a , b >$, of which
$\{1,d_1,d_2,d_3\}$ is a good basis. Furthermore, if $v_2(\mu)=1$ and
$v_2(b)\geq1$, then $\O_2'$ is a maximal suborder of $\O_2$, the index of
$(\O_2')^\times$ in $\O_2^\times$ is $q$, and a set of representatives for the
set $(\O_2')^\times\backslash\O_2^\times$ is given by $\{1+ a_i e_3 :1\leq i
\leq q\}$.

\end{prop}

\begin{prop}\label{prop:general_3}
Suppose that $\O_2$ is in correspondence with the form $f=\< 1 , a , b >$. Let
$a',b'\in \OO_2$. Assume that there exist $\alpha_1,\alpha_2,\alpha_3\in \OO_2$
such that $ab\alpha_1^2=b'$, and $a\alpha_3^2+ b\alpha_2^2=a'$. Let $d_2=
\alpha_1 e_1, d_3= \alpha_2 e_2+\alpha_3 e_3, d_1=d_3 d_2$.

Then, $\O_2'=\< 1,d_1,d_2,d_3>_{\OO_2}$ is a suborder of $\O_2$ in
correspondence with the form $g=\< 1 , a' , b' >$, of which $\{1,d_1,d_2,d_3\}$
is a good basis. Furthermore, if $v_2(b')=v_2(b)+1,v_2(a)=v_2(a')=1$ and
$v_2(b)\geq1$, then $\O_2'$ is a maximal suborder of $\O_2$, the index of
$(\O_2')^\times$ in $\O_2^\times$ is $q$, and a set of representatives for the
set $(\O_2')^\times\backslash\O_2^\times$ is given by $\{1+ a_i e_3 :1\leq i
\leq q\}$.

\end{prop}

Assume that the given system of representatives for $k_2$ is such that $a_1=1$,
and that $a_{q-1}$ and $a_q$ are the two solutions in $k_2$ of $t^2+t+1=0$, when
$q=2^s$ with even $s$.

\begin{prop}\label{prop:repr_unidades2}
Let $\B=\{1,e_1,e_2,e_3\}$ be a quasi-good basis of $\O_2$, and assume that
$\O_2'$ is a maximal suborder of $\O_2$ that has been built using Algorithm
\ref{algo:suborder}. Then, Table \ref{table:repr_unidades2} gives the index of
$(\O_2')^\times$ in $\O_2^\times$ and a system of representatives for the
quotient set.

\begin{table}[h]
\scalebox{0.74}{
\begin{tabular}{|c|c|c|c|c|}
\hline
$\O_2$-class & $\O_2'$-class & $[\O_2^\times:(\O_2')^\times]$ & Representatives
& Condition\\
\hline\hline
\multirow{6}{*}{A1} & \multirow{2}{*}{A1} & $q+1$ & $e_1+e_2, 1 + a_i
e_2\quad (1\leq i \leq q)$ & $s=0$\\
\cline{3-5}
&  & $q$ & $1 + a_i e_2 \quad (1\leq i \leq q)$ & $s\geq1$\\
\cline{2-5}
 & \multirow{2}{*}{A2} & $q(q-1)$ & $(1+a_i e_2)(e_1 +a_j e_2 )\quad (1\leq
i,j\leq q,
a_j\neq0)$ & $r$ odd\\
\cline{3-5}
 &  & $q(q+1)$ & $(1+a_i e_2)(e_1 +a_j e_2 )\quad (1\leq i,j\leq q, a_j\neq0)$,
& $r$ even\\
& & & $(1+a_i e_2)(a_j+e_1)\quad(1\leq i \leq q,q-2\leq j \leq q)$  & \\
\cline{2-5}
 & $B$ & $q-1$ & $1+a_i e_2 \quad (1< i \leq q)$& \\

\hline

\multirow{3}{*}{A2} & A2 & $q^2$ & $1 + a_i e_1 + a_j e_2 \quad (1\leq i,j
\leq q)$ &\\
\cline{2-5}
 & \multirow{2}{*}{B} & $q-1$ & $e_3,1+a_i e_3 \quad (1\leq i\leq q-2)$&
$r$ even\\
\cline{3-5}
 &  & $q+1$ & $e_3,1+a_i e_3 \quad (1\leq i\leq q)$& $r$ odd\\

\hline
\multirow{4}{*}{B} & \multirow{2}{*}{B} &  \multirow{2}{*}{$q$} & $e_2,1+a_i
e_2 \quad (1< i\leq q)$ & $s=0$\\
\cline{4-5}
 & & & $1+a_i e_2  \quad (1\leq i\leq q)$ & $s\geq1$\\
\cline{2-5}
 & C & $q$ & $1+a_i e_3  \quad (1\leq i\leq q)$ & \\
\cline{2-5}
 & D & $q$ & $1+a_i e_2  \quad (1\leq i\leq q)$ & \\
\hline

\multirow{5}{*}{C} & C & $q$ & $1,a_i+ e_2 \quad (1\leq i\leq q)$&  \\

\cline{2-5}

 & \multirow{2}{*}{E} & \multirow{2}{*}{$q$} & $1,a_i+ e_2 \quad (1\leq i\leq
q)$& $\delta_1=1$ \\
\cline{4-5}

 & & & $1,a_i+ e_3 \quad (1\leq i\leq q)$& $\delta_1=3$ \\
\cline{2-5}

 & \multirow{2}{*}{F} & \multirow{2}{*}{$q$} & $1,a_i+ e_2 \quad (1\leq i\leq
q)$& $\delta_1=1$ \\
\cline{4-5}

 &  &  & $1,a_i+ e_3 \quad (1\leq i\leq q)$& $\delta_1=3$ \\

\hline

D & D & $q$ & $1,a_i+ e_2 \quad (1\leq i\leq q)$&  \\

\hline

\multirow{2}{*}{E} & E & $q$ & $1,a_i+ e_2 \quad (1\leq i\leq q)$&  \\

\cline{2-5}

& G & $q$ & $1,a_i+ e_3 \quad (1\leq i\leq q)$&  \\

\hline

F & F & $q$ & $1,a_i+ e_2 \quad (1\leq i\leq q)$&  \\

\hline

G & G & $q$ & $1,a_i+ e_2 \quad (1\leq i\leq q)$&  \\

\hline

\end{tabular}}
\caption{$[\O_2^\times:(\O_2')^\times]$ and representatives for
$(\O_2')^\times\backslash\O_2^\times$}\label{table:repr_unidades2}
\end{table}

\end{prop}

\begin{proof}

As in the $\pp\nmid2$ case, by Proposition \ref{prop:unid_mod_p}, we may assume
that $\B$ is a good basis for $\O_2$, as well as we may perform all calculations
modulo $2\O_2$.

The cases B to B, C to C, D to D, E to E, F to F and G to G are covered by
Proposition \ref{prop:general_1}. The case B to C is covered by Proposition
\ref{prop:general_2}.

To prove the case B to D, use Proposition \ref{prop:general_2} to descend from
$\< 1 , 1 , 2^2 > $ to $\< 1 , 5 , 2^2> $, and Proposition \ref{prop:general_1}
to descend from this form to $\< 1 , 5 , 3\cdot 2^3 > $. A similar argument
works for the other form of class B.

The cases  C to E (with $\delta_1=3$), C to F (with $\delta_1=3$) and E to G are
covered by Proposition \ref{prop:general_3}.

Now we will prove the case from A2 to B. The remaining cases can be treated in
a similar way, with no further difficulties. 

By (\ref{eqn:norma2}), the norm form on $2 \O_2\backslash\O_2$ is given by
$N(x)=x_0^2+x_0 x_3+ x_3^2$. Hence, a standard calculation shows that
\[
 \vert (2 \O_2\backslash\O_2)^\times \vert=\begin{cases}
                                              q^4-q^2(2q-1), & \textrm{if } r
\textrm{ is even } \\
q^4-q^2, & \textrm{if } r \textrm{ is odd } 
                                             \end{cases}
\]

Since $d_1 = 1+ e_1, d_2= 1+e_2$ and $d_3= 1+ e_1 + e_2$ in $2
\O_2\backslash\O_2$, we have that $2 \O_2\backslash\O'_2=\< 1,e_1,e_2>_{k_2}$.
Hence $\vert(2 \O_2\backslash\O_2)^\times\vert=q^3-q^2$, and this proves the
equality on $[\O_2^\times:(\O_2')^\times]$.

Now we need to find the right amount of non equivalent units. It is easily seen
that the elements in the set $\{1+a_i e_3:1\leq i\leq q\}\cup\{e_3\}$ are not
mutually equivalent modulo $(2 \O_2\backslash\O_2)^\times$, and they are all
units, except for $1+a_{q-1} e_3$ and $1+a_q e_3$ when $q=2^s$ with even $s$.

\end{proof}


\begin{thebibliography}{EMP86}

\bibitem[Brz82]{brz-gor}
Juliusz Brzezinski.
\newblock A characterization of {G}orenstein orders in quaternion algebras.
\newblock {\em Math. Scand.}, 50(1):19--24, 1982.

\bibitem[Brz83]{brz-ord}
Juliusz Brzezinski.
\newblock On orders in quaternion algebras.
\newblock {\em Comm. Algebra}, 11(5):501--522, 1983.

\bibitem[Brz90]{brz-aut}
Juliusz Brzezinski.
\newblock On automorphisms of quaternion orders.
\newblock {\em J. Reine Angew. Math.}, 403:166--186, 1990.

\bibitem[Coh00]{cohen}
Henri Cohen.
\newblock {\em Advanced topics in computational number theory}, volume 193 of
  {\em Graduate Texts in Mathematics}.
\newblock Springer-Verlag, New York, 2000.

\bibitem[CS01]{consani}
Caterina Consani and Jasper Scholten.
\newblock Arithmetic on a quintic threefold.
\newblock {\em Internat. J. Math.}, 12(8):943--972, 2001.

\bibitem[DD08]{lassina}
Lassina Demb{\'e}l{\'e} and Steve Donnelly.
\newblock Computing {H}ilbert modular forms over fields with nontrivial class
  group.
\newblock In {\em Algorithmic number theory}, volume 5011 of {\em Lecture Notes
  in Comput. Sci.}, pages 371--386. Springer, Berlin, 2008.

\bibitem[EMP86]{edgar}
H.~M. Edgar, R.~A. Mollin, and B.~L. Peterson.
\newblock Class groups, totally positive units, and squares.
\newblock {\em Proc. Amer. Math. Soc.}, 98(1):33--37, 1986.

\bibitem[GL09]{gross-luc}
Benedict~H. Gross and Mark~W. Lucianovic.
\newblock On cubic rings and quaternion rings.
\newblock {\em J. Number Theory}, 129(6):1468--1478, 2009.

\bibitem[Kap69]{kap}
Irving Kaplansky.
\newblock Submodules of quaternion algebras.
\newblock {\em Proc. London Math. Soc. (3)}, 19:219--232, 1969.

\bibitem[KV10]{kirschmer}
Markus Kirschmer and John Voight.
\newblock Algorithmic enumeration of ideal classes for quaternion orders.
\newblock {\em SIAM J. Comput.}, 39(5):1714--1747, 2010.

\bibitem[Lem11]{joh}
Stefan Lemurell.
\newblock Quaternion orders and ternary quadratic forms.
\newblock 2011.
\newblock \url{http://arxiv.org/abs/1103.4922}.

\bibitem[Piz76]{piz-arithII}
Arnold Pizer.
\newblock On the arithmetic of quaternion algebras. {II}.
\newblock {\em J. Math. Soc. Japan}, 28(4):676--688, 1976.

\bibitem[Piz80]{Pizer}
Arnold Pizer.
\newblock An algorithm for computing modular forms on {$\Gamma _{0}(N)$}.
\newblock {\em J. Algebra}, 64(2):340--390, 1980.

\bibitem[PRV05]{villegas}
Ariel Pacetti and Fernando Rodriguez~Villegas.
\newblock Computing weight 2 modular forms of level {$p^2$}.
\newblock {\em Math. Comp.}, 74(251):1545--1557 (electronic), 2005.
\newblock With an appendix by B. Gross.

\bibitem[PT07]{tornaria}
Ariel Pacetti and Gonzalo Tornar{\'{\i}}a.
\newblock Shimura correspondence for level {$p^2$} and the central values of
  {$L$}-series.
\newblock {\em J. Number Theory}, 124(2):396--414, 2007.

\bibitem[Sa11]{sage}
W.\thinspace{}A. Stein et~al.
\newblock {\em {S}age {M}athematics {S}oftware ({V}ersion 4.7.2)}.
\newblock The Sage Development Team, 2011.
\newblock \url{http://www.sagemath.org}.


\bibitem[SW05]{socrates}
Jude Socrates and David Whitehouse.
\newblock Unramified {H}ilbert modular forms, with examples relating to
  elliptic curves.
\newblock {\em Pacific J. Math.}, 219(2):333--364, 2005.

\bibitem[Vig76]{vigneras-simpl}
Marie-France Vign{\'e}ras.
\newblock Simplification pour les ordres des corps de quaternions totalement
  d\'efinis.
\newblock {\em J. Reine Angew. Math.}, 286/287:257--277, 1976.

\bibitem[Vig80]{vig}
Marie-France Vign{\'e}ras.
\newblock {\em Arithm\'etique des alg\`ebres de quaternions}, volume 800 of
  {\em Lecture Notes in Mathematics}.
\newblock Springer, Berlin, 1980.


\bibitem[Voi10]{voight}
John Voight.
\newblock Identifying the matrix ring: algorithms for quaternion algebras and
  quadratic forms.
\newblock 2010.
\newblock \url{http://arxiv.org/abs/1004.0994}

\end{thebibliography}
\end{document}